\documentclass[12pt,thmsa]{article}
\usepackage{amssymb}
\usepackage{amsfonts}
\usepackage{amsmath}
\newtheorem{theorem}{Theorem}[section]
\newtheorem{proposition}[theorem]{Proposition}
\newtheorem{corollary}[theorem]{Corollary}

\newtheorem{lemma}[theorem]{Lemma}

\newtheorem{definition}[theorem]{Definition}
\newenvironment{proof}[1][Proof]{\noindent\textbf{#1.} }{\hfill $\square$}

\begin{document}

\title{Bound state nodal solutions for the non-autonomous Schr\"{o}%
dinger--Poisson system in $\mathbb{R}^{3}$}
\author{Juntao Sun$^{a}$\thanks{%
E-mail address: sunjuntao2008@163.com(J. Sun)}, Tsung-fang Wu$^{b}$\thanks{%
E-mail address: tfwu@nuk.edu.tw (T.-F. Wu)} \\
%EndAName
{\footnotesize $^a$\emph{School of Mathematics and Statistics, Shandong
University of Technology, Zibo, 255049, P.R. China }}\\
{\footnotesize $^b$\emph{Department of Applied Mathematics, National
University of Kaohsiung, Kaohsiung 811, Taiwan }}}
\date{}
\maketitle

\begin{abstract}
In this paper, we study the existence of nodal solutions for the
non-autonomous Schr\"{o}dinger--Poisson system:
\begin{equation*}
\left\{
\begin{array}{ll}
-\Delta u+u+\lambda K(x) \phi u=f(x) |u|^{p-2}u & \text{ in }\mathbb{R}^{3},
\\
-\Delta \phi =K(x)u^{2} & \text{ in }\mathbb{R}^{3},%
\end{array}%
\right.
\end{equation*}%
where $\lambda >0$ is a parameter and $2<p<4$. Under some proper assumptions
on the nonnegative functions $K(x)$ and $f(x)$, but not requiring any
symmetry property, when $\lambda$ is sufficiently small, we find a bounded
nodal solution for the above problem by proposing a new approach, which
changes sign exactly once in $\mathbb{R}^{3}$. In particular, the existence
of a least energy nodal solution is concerned as well.
\end{abstract}

\section{Introduction}

Consider the non-autonomous Schr\"{o}dinger--Poisson system in the form:%
\begin{equation}
\left\{
\begin{array}{ll}
-\Delta u+u+\lambda K(x)\phi u=f\left( x\right) |u|^{p-2}u & \text{ in }%
\mathbb{R}^{3}, \\
-\Delta \phi =K(x)u^{2} & \text{ in }\mathbb{R}^{3},%
\end{array}%
\right.  \tag{$SP_{\lambda }$}
\end{equation}%
where $\lambda >0,$ $2<p<4$ and the functions $f(x)$ and $K(x)$ satisfy the
following assumptions:

\begin{itemize}
\item[$\left( F1\right) $] $f(x)$ is a positive continuous function on $%
\mathbb{R}^{3}$ such that
\begin{equation*}
\lim_{\left\vert x\right\vert \rightarrow \infty }f\left( x\right)
=f_{\infty }>0\text{ uniformly on}\ \mathbb{R}^{3},
\end{equation*}%
and
\begin{equation*}
f_{\max }:=\sup_{x\in \mathbb{R}^{3}}f\left( x\right) <\frac{f_{\infty }}{%
A\left( p\right) ^{\frac{p-2}{2}}},
\end{equation*}%
where%
\begin{equation*}
A\left( p\right) =\left\{
\begin{array}{ll}
\left( \frac{4-p}{2}\right) ^{\frac{1}{p-2}}, & \text{ if }2<p\leq 3, \\
\frac{1}{2}, & \text{ if }3<p<4.%
\end{array}%
\right.
\end{equation*}

\item[$\left( K1\right) $] $K(x)\in L^{\infty }(\mathbb{R}^{3})\backslash
\left\{ 0\right\} $ is a nonnegative function on $\mathbb{R}^{3}.$
\end{itemize}

In quantum mechanics, Schr\"{o}dinger--Poisson systems (SP systems for
short), of the form similar to system $(SP_{\lambda }),$ can be used to
describe the interaction of a charged particle with the electrostatic field.
Indeed, the unknowns $u$ and $\phi $ represent the wave functions associated
with the particle and the electric potentials, respectively. The function $%
K(x)$ denotes a nonnegative density charge, and the local nonlinearity $%
f\left( x\right) |u|^{p-2}u$ (or, more generally, $g(x,u)$) simulates the
interaction effect among many particles. For more details about its physical
meaning, we refer the reader to \cite{BF} and the references therein.

It is well-known that SP systems can be transformed into the nonlinear Schr%
\"{o}dinger equations with a non-local term \cite{BF,CG,R}. Using system $%
(SP_{\lambda })$ as an example, it becomes the following equation%
\begin{equation}
\begin{array}{ll}
-\Delta u+u+\lambda K\left( x\right) \phi _{K,u}u=f\left( x\right)
\left\vert u\right\vert ^{p-2}u & \text{ in }\mathbb{R}^{3},%
\end{array}
\tag*{$\left( E_{\lambda }\right) $}
\end{equation}%
where $\phi _{K,u}(x)=\frac{1}{4\pi }\int_{\mathbb{R}^{3}}\frac{K(y)}{%
\left\vert x-y\right\vert }u^{2}(y)dy.$ Eq. $\left( E_{\lambda }\right) $ is
variational, and its solutions are the critical points of the energy
functional $I_{\lambda }\left( u\right) $ defined in $H^{1}(\mathbb{R}^{3})$
by
\begin{equation*}
I_{\lambda }\left( u\right) =\frac{1}{2}\left\Vert u\right\Vert _{H^{1}}^{2}+%
\frac{\lambda }{4}\int_{\mathbb{R}^{3}}K(x)\phi _{K,u}u^{2}dx-\frac{1}{p}%
\int_{\mathbb{R}^{3}}f(x)\left\vert u\right\vert ^{p}dx,
\end{equation*}%
where $\left\Vert u\right\Vert _{H^{1}}=\left[ \int_{\mathbb{R}^{3}}\left(
\left\vert \nabla u\right\vert ^{2}+u^{2}\right) dx\right] ^{1/2}$ is the
standard norm in $H^{1}(\mathbb{R}^{3}).$ In view of this, variational
methods have been effective tools in finding nontrivial solutions of SP
systems.

In recent years, there has been much attention to SP systems like system $%
(SP_{\lambda })$ on the existence of positive solutions, ground states,
radial solutions and semiclassical states. We refer the reader to \cite%
{AR,AP,CM,CV,CKW,IR,MT,M,R,SCN,SWF,SWF1,ZZ}. More precisely, Ruiz \cite{R}
studied the autonomous SP system%
\begin{equation}
\left\{
\begin{array}{ll}
-\Delta u+u+\lambda \phi u=|u|^{p-2}u & \text{ in }\mathbb{R}^{3}, \\
-\Delta \phi =u^{2} & \text{ in }\mathbb{R}^{3}.%
\end{array}%
\right.  \label{1-4}
\end{equation}%
In order to find nontrivial solutions of system $(\ref{1-4})$ with $2<p<6,$
a Nehari-Pohozaev manifold is constructed, with the aid of the Pohozaev
identity corresponding to system $(\ref{1-4}).$ As a consequence, for $%
\lambda >0$ sufficiently small, two positive radial solutions and one
positive radial solution have been obtained when $2<p<3$ and $3\leq p<6,$
respectively. Moreover, when $\lambda \geq \frac{1}{4},$ it has been shown
that $p=3$ is a critical value for the existence of nontrivial solutions.
The corresponding results have been further improved by Azzollini-Pomponio
\cite{AP} by showing the existence of ground state solutions when $\lambda
>0 $ and $3<p<6.$

Cerami and Varia \cite{CV} studied a class of non-autonomous SP systems
without any symmetry assumptions, i.e., system $(SP_{\lambda })$ with $%
\lambda =1.$ By establishing the compactness lemma and using the Nehari
manifold, when $K(x)$ and $f(x)$ satisfy some suitable assumptions, the
existence of positive ground state and bound state solutions have been
proved for $4<p<6.$ Later, when the mass term $u$ is replaced by $V(x)u$ in
system $(SP_{\lambda }),$ by assuming the decay rate of the coefficients $%
V(x),K(x)$ and $f(x)$, Cerami and Molle \cite{CM} obtained the existence of
positive bound state solution for system $(SP_{\lambda })$ with $\lambda =1$
and $4<p<6$ via the Nehari manifold, which complements the result in \cite%
{CV} in some sense.

Very recently, we \cite{SWF1} investigated the existence of a positive
solution for system $(SP_{\lambda })$ with $2<p<4$ when $\lambda $ is
sufficiently small. Distinguishing from the case of $4\leq p<6,$ we notice
that in this case the (PS)--sequences for the energy functional $I_{\lambda
} $ may not be bounded and $I_{\lambda }(tu)\rightarrow \infty $ as $%
t\rightarrow \infty $ for each $u\in H^{1}(\mathbb{R}^{3})\backslash \left\{
0\right\} $. So variational methods cannot be applied in a standard way,
even restricting $I_{\lambda }$ on the Nehari manifold. Moreover, the
Nehari-Pohozaev manifold presented by Ruiz is also not a ideal choice for
the non-autonomous system like system $(SP_{\lambda })$, since the Pohozaev
identity corresponding to system $(SP_{\lambda })$ is extremely complicated.
For these reasons, in \cite{SWF1} we introduced a filtration of the Nehari
manifold $\mathbf{M}_{\lambda }$ as follows%
\begin{equation*}
\mathbf{M}_{\lambda }(c)=\{u\in \mathbf{M}_{\lambda }:I_{\lambda }(u)<c\}%
\text{ for some }c>0,
\end{equation*}%
and showed that this set $\mathbf{M}_{\lambda }(c)$ under the given
assumptions is the union of two disjoint nonempty sets, namely,
\begin{equation*}
\mathbf{M}_{\lambda }(c)=\mathbf{M}_{\lambda }^{(1)}\cup \mathbf{M}_{\lambda
}^{(2)},
\end{equation*}%
which are both $C^{1}$ sub-manifolds of $\mathbf{M}_{\lambda }$ and natural
constraints of $I_{\lambda }.$ Moreover, $\mathbf{M}_{\lambda }^{(1)}$ is
bounded such that $I_{\lambda }$ is coercive and bounded below on it,
whereas $I_{\lambda }$ is unbounded below on $\mathbf{M}_{\lambda }^{(2)}.$
In fact, $\mathbf{M}_{\lambda }^{(2)}$ may not contain any non-zero critical
point of $I_{\lambda }$ for $\frac{1+\sqrt{73}}{3}<p<4$ (see \cite[Theorem
1.6]{SWF1}). Thus, our approach is seeking a minimizer of $I_{\lambda }$ on
the constraint $\mathbf{M}_{\lambda }^{(1)}.$

Another topic which has received increasingly interest of late years is the
existence of nodal (or sign-changing) solutions for SP systems, see, for
example, \cite{AS,ASS,BF1,CT,I,KS,LXZ,LWZ,SW,WZ}. Recall that a solution $%
(u,\phi )$ to SP systems is called a nodal solution if $u$ changes sign,
i.e., $u^{\pm }\not\equiv 0,$ where
\begin{equation*}
u^{+}(x)=\max \{u(x),0\}\text{ and }u^{-}(x)=\min \{u(x),0\}.
\end{equation*}%
By using the Nehari manifold and gluing solution pieces together, Kim and
Seok \cite{KS} proved the existence of a radial nodal solution with
prescribed numbers of nodal domains for system $(\ref{1-4})$ with $\lambda
>0 $ and $4<p<6$. Almost simultaneously, a similar result to \cite{KS} for $%
4\leq p<6$ has been established by Ianni \cite{I} via a dynamical approach
together with a limit procedure. Of particular note is that all nodal
solutions found in \cite{I,KS} have certain types of symmetries, and thus
the system is required to have a certain group invariance.

In \cite{WZ}, Wang and Zhou studied the following non-autonomous SP system
without any symmetry%
\begin{equation}
\left\{
\begin{array}{ll}
-\Delta u+V(x)u+\phi u=\left\vert u\right\vert ^{p-2}u & \text{ in }\mathbb{R%
}^{3}, \\
-\Delta \phi =u^{2} & \ \text{in }\mathbb{R}^{3}.%
\end{array}%
\right.  \label{1-6}
\end{equation}%
By using the nodal Nehari manifold
\begin{equation*}
\mathbf{N}=\left\{ u\in H\text{ }|\text{ }\left\langle I^{\prime
}(u),u^{+}\right\rangle =\left\langle I^{\prime }(u),u^{-}\right\rangle =0%
\text{ and }u^{\pm }\neq 0\right\}
\end{equation*}%
as well as the Brouwer degree theory, the existence of a least energy nodal
solution for system $(\ref{1-6})$ with $4<p<6$ has been proved when either $%
V(x)$ is a positive constant or $V(x)\in C(\mathbb{R}^{3},\mathbb{R}^{+})$
such that $H\subset H^{1}(\mathbb{R}^{3})$ and the embedding $%
H\hookrightarrow L^{q}(\mathbb{R}^{3})(2<p<6)$ is compact. Applying the same
approach, some similar results to \cite{WZ} have been obtained in \cite%
{AS,ASS,BF1,CT,LXZ,SW} when the nonlinearity is either $g(x,u)$ or $%
f(x)\left\vert u\right\vert ^{p-2}u(4\leq p<6)$. Note that such a $g(x,u)$
is merely a general form of $f(x)\left\vert u\right\vert ^{p-2}u(4\leq p<6)$%
, not covering the case of $2<p<4.$

In \cite{LWZ}, Liu, Wang and Zhang proved the existence of infinitely many
nodal solutions for system $(\ref{1-6})$ with $3<p<6$ when $V(x)$ is
coercive in $\mathbb{R}^{3}$ for recovering the compactness. The proof is
mainly based on the method of invariant sets of descending flow.
Furthermore, in the case of $3<p<4,$ a perturbation approach is also used by
constructing an auxiliary system and passing the limit to the original one.

To the best of our knowledge, there seems no result in the existing
literature on nodal solutions of SP systems in the case of $2<p<4$, except
\cite{LWZ}. Inspired by this fact, in the present paper we are interested in
the existence of a nodal solution for a class of non-autonomous SP systems
when the nonlinearity is like $f(x)\left\vert u\right\vert ^{p-2}u(2<p<4)$,
i.e., system $(SP_{\lambda })$ with $2<p<4.$ It is worth emphasizing that in
this case the existence of a least energy nodal solution is concerned as
well.

We wish to point out that the approaches in \cite%
{AS,ASS,BF1,CT,I,KS,LXZ,SW,WZ} are only valid for the case of $4\leq p<6,$
and that the approach in \cite{LWZ} can only solve the case of $3<p<6.$

In this study, following a part of the idea in our recent paper \cite{SWF1},
we propose a new approach to seek nodal solutions of system $(SP_{\lambda })$
with $2<p<4.$ That is, we construct a nonempty nodal set $\mathbf{N}%
_{\lambda }^{(1)}$ in the bounded set $\mathbf{M}_{\lambda }^{(1)}$
introduced in \cite{SWF1}, where $I_{\lambda }$ is coercive and bounded
below, and then minimize $I_{\lambda }$ on it, not on the nodal Nehari
manifold $\mathbf{N.}$ In fact, such a $\mathbf{N}_{\lambda }^{(1)}$ is a
subset of $\mathbf{N}.$

In analysis, we have to face several challenges. First of all, note that the
nodal set\ $\mathbf{N}_{\lambda }^{(1)}$ is not manifold. Then one cannot
talk about vector fields on $\mathbf{N}_{\lambda }^{(1)}$ and one cannot
easily construct deformations on $\mathbf{N}_{\lambda }^{(1)}.$ As a
consequence, min-max values for $I_{\lambda }$ on $\mathbf{N}_{\lambda
}^{(1)}$ are not automatically critical points of $I_{\lambda }.$ In fact, $%
\mathbf{N}_{\lambda }^{(1)}\cap H^{2}(\mathbb{R}^{3})$ are codimension $2$
submanifolds of $H^{2}(\mathbb{R}^{3})$ (see \cite{BW2,BW}). Secondly, since
$\mathbf{N}_{\lambda }^{(1)}$ is just a subset of the nodal Nehari manifold $%
\mathbf{N}$, it seems not easy to show that $\mathbf{N}_{\lambda }^{(1)}\neq
\emptyset ,$ which has never been involved before. Thirdly, for each $u\in
H^{1}(\mathbb{R}^{3})$ with $u^{\pm }\not\equiv 0,$ the function $\widetilde{%
h}(s,t)=I_{\lambda }\left( su^{+}+tu^{-}\right) $ is not strictly concave on
$(0,\infty )\times (0,\infty )$ when $2<p<4,$ which is totally different
from the case of $4\leq p<6$. Finally, we notice that the decomposition%
\begin{equation*}
\int_{\mathbb{R}^{3}}K(x)\phi _{K,u}u^{2}dx=\int_{\mathbb{R}^{3}}K(x)\phi
_{K,u^{+}}\left( u^{+}\right) ^{2}dx+\int_{\mathbb{R}^{3}}K(x)\phi
_{K,u^{-}}\left( u^{-}\right) ^{2}dx
\end{equation*}%
does not hold in general, making the problem more complicated. In order to
overcome these difficulties, in this paper some new ideas are introduced and
some new estimates are established.

\begin{definition}
$(u,\phi )$ is called a least energy nodal solution of system $(SP_{\lambda
})$, if $(u,\phi )$ is a solution of system $(SP_{\lambda })$ which has the
least energy among all nodal solutions of system $(SP_{\lambda })$.
\end{definition}

We now summarize our main results as follows.

\begin{theorem}
\label{t2}Suppose that $2<p<4,$ and conditions $(F1)$ and $(K1)$ hold. In
addition, we assume that

\begin{itemize}
\item[$(F2)$] there exists $0<r_{f}<1$ such that $f\left( x\right) \geq
f_{\infty }+d_{0}\exp \left( -\left\vert x\right\vert ^{r_{f}}\right) $ for
some $d_{0}>0$ and for all $x\in \mathbb{R}^{3};$

\item[$(K2)$] $K\left( x\right) \lneqq K_{\infty }$ for all $x\in \mathbb{R}%
^{3}$ and $\lim_{\left\vert x\right\vert \rightarrow \infty }K\left(
x\right) =K_{\infty }>0\ $uniformly on$\ \mathbb{R}^{3}.$
\end{itemize}

Then there exists $\lambda ^{\ast }>0$ such that for every $0<\lambda
<\lambda ^{\ast },$ system $(SP_{\lambda })$ admits a nodal solution $%
(u_{\lambda },\phi _{K,u_{\lambda }})\in H^{1}(\mathbb{R}^{3})\times D^{1,2}(%
\mathbb{R}^{3}),$ which changes sign exactly once in $\mathbb{R}^{3}.$
Furthermore, there holds
\begin{equation*}
\left( \frac{S_{p}^{p}}{f_{\max }}\right) ^{\frac{1}{p-2}}\leq \left\Vert
u_{\lambda }^{\pm }\right\Vert _{H^{1}}<\left( \frac{2S_{p}^{p}}{\left(
4-p\right) f_{\max }}\right) ^{\frac{1}{p-2}},
\end{equation*}%
and
\begin{equation*}
\left\Vert \phi _{K,u_{\lambda}}\right\Vert _{D^{1,2}}\leq\overline{S}%
^{-1}S_{12/5}^{-2}K_{\max }\left( \frac{2S_{p}^{p}}{\left( 4-p\right)
f_{\max }}\right) ^{\frac{2}{p-2}},
\end{equation*}%
where $S_{p}$ is the best constant for the embedding of $H^{1}(\mathbb{R}%
^{3})$ in $L^{p}(\mathbb{R}^{3})$ with $2<p<4,$ $\overline{S}$ is the best
constant for the embedding of $D^{1,2}(\mathbb{R}^{3})$ in $L^{6}(\mathbb{R}%
^{3}),$ and $S_{12/5}=S_{p}$ with $p=12/5$.
\end{theorem}

\begin{theorem}
\label{t3}Suppose that $2<p<4$ and conditions ${(F1)}-{(F2)}$ and $\left(
K1\right) $ hold. In addition, we assume that

\begin{itemize}
\item[$\left( K3\right) $] $K(x)\in L^{2}(\mathbb{R}^{3})$ and $%
\lim_{\left\vert x\right\vert \rightarrow \infty }K\left( x\right) =0.$
\end{itemize}

Then there exists $\overline{\lambda }^{\ast }>0$ such that for each $%
0<\lambda <\overline{\lambda }^{\ast },$ system $(SP_{\lambda })$ admits a
nodal solution $(u_{\lambda },\phi _{K,u_{\lambda }})\in H^{1}(\mathbb{R}%
^{3})\times D^{1,2}(\mathbb{R}^{3}),$ which changes sign exactly once in $%
\mathbb{R}^{3}.$ Furthermore, there holds
\begin{equation*}
\left( \frac{S_{p}^{p}}{f_{\max }}\right) ^{\frac{1}{p-2}}\leq \left\Vert
u_{\lambda }^{\pm }\right\Vert _{H^{1}}<\left( \frac{2S_{p}^{p}}{\left(
4-p\right) f_{\max }}\right) ^{\frac{1}{p-2}},
\end{equation*}
and
\begin{equation*}
\left\Vert \phi _{K,u_{\lambda}}\right\Vert _{D^{1,2}}\leq\overline{S}%
^{-1}S_{12/5}^{-2}K_{\max }\left( \frac{2S_{p}^{p}}{\left( 4-p\right)
f_{\max }}\right) ^{\frac{2}{p-2}}.
\end{equation*}
\end{theorem}

According to \cite[Theorem 1.6]{SWF1}, we have the following theorem on the
existence of a least energy nodal solution.

\begin{theorem}
\label{t4}Suppose that $\frac{1+\sqrt{73}}{3}<p<4,$ and conditions ${(F1)}$
and ${(K1)}$ hold. In addition, we assume that

\begin{itemize}
\item[$\left( D_{K,f}\right) $] the functions $f(x),K(x)\in C^{1}(\mathbb{R}%
^{3})$ satisfy $\langle \nabla f(x),x\rangle \leq 0$ and%
\begin{equation*}
\frac{3p^{2}-2p-24}{2(6-p)}K(x)+\frac{p-2}{2}\langle \nabla K(x),x\rangle
\geq 0.
\end{equation*}
\end{itemize}

If $(u_{\lambda },\phi _{K,u_{\lambda }})$ is the nodal solution as
described in Theorem \ref{t2} or \ref{t3}, then $(u_{\lambda },\phi
_{K,u_{\lambda }})$ is a least energy nodal solution of system $(SP_{\lambda
}).$
\end{theorem}

This paper is organized as follows. After introducing various preliminaries
in Section 2, we give the estimates of energy and construct the
Palais--Smale sequences in Sections 3 and 4, respectively. In Sections 5 and
6, we prove Theorems \ref{t2} and \ref{t3}, respectively.

\section{Preliminaries}

As pointed out in the section of Introduction, system $(SP_{\lambda })$ can
be transferred into a nonlocal Schr\"{o}dinger equation, i.e., Eq. $\left(
E_{\lambda }\right) $, and its corresponding energy functional is $%
I_{\lambda }(u)$. It is not difficult to prove that $I_{\lambda }$ is a $%
C^{1}$ functional with the derivative given by%
\begin{equation*}
\left\langle I_{\lambda }^{\prime }(u),\varphi \right\rangle =\int_{\mathbb{R%
}^{3}}\left( \nabla u\nabla \varphi +u\varphi +\lambda K(x)\phi
_{K,u}u\varphi -f(x)|u|^{p-2}u\varphi \right) dx
\end{equation*}%
for all $\varphi \in H^{1}(\mathbb{R}^{3}),$ where $I_{\lambda }^{\prime }$
is the Fr\'{e}chet derivative of $I_{\lambda }.$ Note that $(u,\phi )\in
H^{1}(\mathbb{R}^{3})\times D^{1,2}(\mathbb{R}^{3})$ is a solution of system
$\left( SP_{\lambda }\right) $ if and only if $u$ is a critical point of $%
I_{\lambda }$ and $\phi =\phi _{K,u}$.

Next, we give a characterization of the weak convergence for the Poisson
term. The proof can be made in a similar argument as in \cite{IR}.

\begin{lemma}
\label{h2}Suppose that condition $(K1)$ holds. Define the operator $\Pi :%
\left[ H^{1}(\mathbb{R}^{3})\right] ^{4}\rightarrow \mathbb{R}$ by%
\begin{equation*}
\Pi \left( u,v,w,z\right) =\int_{\mathbb{R}^{3}}\int_{\mathbb{R}^{3}}\frac{%
K(x)K\left( y\right) }{\left\vert x-y\right\vert }u\left( x\right) v\left(
x\right) w\left( y\right) z\left( y\right) dxdy
\end{equation*}%
for all $\left( u,v,w,z\right) \in \left[ H^{1}(\mathbb{R}^{3})\right] ^{4}.$
Then for all $\{u_{n}\},\{v_{n}\},\{w_{n}\}\subset H^{1}(\mathbb{R}^{3})$
satisfying $u_{n}\rightharpoonup u$ in $H^{1}(\mathbb{R}^{3}),v_{n}%
\rightharpoonup v$ in $H^{1}(\mathbb{R}^{3}),w_{n}\rightharpoonup w$ in $%
H^{1}(\mathbb{R}^{3})$ and for all $z\in H^{1}(\mathbb{R}^{3}),$ there holds%
\begin{equation*}
\Pi \left( u_{n},v_{n},w_{n},z\right) \rightarrow \Pi \left( u,v,w,z\right) .
\end{equation*}
\end{lemma}

In the following lemma we summarize some useful properties on the function $%
\phi _{K,u}$, which have been obtained in \cite{AP,CV}.

\begin{lemma}
\label{h1}Suppose that condition $(K1)$ holds. Then for each $u\in H^{1}(%
\mathbb{R}^{3})$, we have the following statements.\newline
$\left( i\right) $ $\left\Vert \phi _{K,u}\right\Vert _{D^{1,2}}\leq
\overline{S}^{-1}S_{12/5}^{-2}K_{\max }\left\Vert u\right\Vert _{H^{1}}^{2}$
holds. As a consequence, there holds%
\begin{equation*}
\int_{\mathbb{R}^{3}}K(x)\phi _{K,v}u^{2}dx\leq \overline{S}%
^{-2}S_{12/5}^{-4}K_{\max }^{2}\Vert v\Vert _{H^{1}}^{2}\Vert u\Vert
_{H^{1}}^{2};
\end{equation*}%
$\left( ii\right)$ Both $\phi _{K,u}\geq 0$ and $\phi _{K,u}\left( x\right) =%
\frac{1}{4\pi }\int_{\mathbb{R}^{3}}\frac{K\left( y\right) u^{2}\left(
y\right) }{\left\vert x-y\right\vert }dy$ hold;\newline
$\left( iii\right) $ For any $t>0,\ \phi _{K,tu}=t^{2}\phi _{K,u}$ holds;%
\newline
$\left( iv\right) $ If $u_{n}\rightharpoonup u$ in $H^{1}(\mathbb{R}^{3}),$
then $\Phi \left[ u_{n}\right] \rightharpoonup \Phi \left[ u\right] $ in $%
D^{1,2}(\mathbb{R}^{3}),$ where the operator $\Phi :H^{1}(\mathbb{R}%
^{3})\rightarrow D^{1,2}(\mathbb{R}^{3})$ as $\Phi \left[ u\right] =\phi
_{K,u};$\newline
$\left( v\right) $ If we, in addition, assume that condition $\left(
K3\right) $ holds, then
\begin{equation*}
\int_{\mathbb{R}^{3}}K(x)\phi _{K,u_{n}}u_{n}^{2}dx\rightarrow \int_{\mathbb{%
R}^{3}}K(x)\phi _{K,u}u^{2}dx\text{ as }n\rightarrow \infty ,
\end{equation*}%
when $u_{n}\rightharpoonup u$ in $H^{1}(\mathbb{R}^{3}).$
\end{lemma}

Define the Nehari manifold
\begin{equation*}
\mathbf{M}_{\lambda }=\left\{ u\in H^{1}(\mathbb{R}^{3})\backslash \{0\}%
\text{ }|\text{ }\left\langle I_{\lambda }^{\prime }\left( u\right)
,u\right\rangle =0\right\} .
\end{equation*}%
Then $u\in \mathbf{M}_{\lambda }$ if and only if
\begin{equation*}
\left\Vert u\right\Vert _{H^{1}}^{2}+\lambda \int_{\mathbb{R}^{3}}K(x)\phi
_{K,u}u^{2}dx-\int_{\mathbb{R}^{3}}f(x)|u|^{p}dx=0.
\end{equation*}%
Moreover, it follows from the Sobolev inequality that
\begin{eqnarray*}
\left\Vert u\right\Vert _{H^{1}}^{2} &\leq &\left\Vert u\right\Vert
_{H^{1}}^{2}+\lambda \int_{\mathbb{R}^{3}}K(x)\phi _{K,u}u^{2}dx \\
&=&\int_{\mathbb{R}^{3}}f(x)\left\vert u\right\vert ^{p}dx\leq
S_{p}^{-p}f_{\max }\left\Vert u\right\Vert _{H^{1}}^{p}\text{ for all }u\in
\mathbf{M}_{\lambda },
\end{eqnarray*}%
this implies\ that
\begin{equation}
\left\Vert u\right\Vert _{H^{1}}\geq \left( \frac{S_{p}^{p}}{f_{\max }}%
\right) ^{\frac{1}{p-2}}\text{ for all }u\in \mathbf{M}_{\lambda },
\label{2}
\end{equation}%
where $S_{p}$ is the best Sobolev constant for the embedding of $H^{1}(%
\mathbb{R}^{3})$ in $L^{p}(\mathbb{R}^{3}).$

As we know, the Nehari manifold $\mathbf{M}_{\lambda }$ is closely related
to the behavior of the function $h_{u}:t\rightarrow I_{\lambda }\left(
tu\right) $ for $t>0.$ Such map is known as fibering map. About its theory
and application, we refer the reader to \cite{BB,BZ,DP,P1,P2}. For $u\in
H^{1}(\mathbb{R}^{3}),$ we have%
\begin{equation*}
h_{u}\left( t\right) =\frac{t^{2}}{2}\left\Vert u\right\Vert _{H^{1}}^{2}+%
\frac{\lambda t^{4}}{4}\int_{\mathbb{R}^{3}}K(x)\phi _{K,u}u^{2}dx-\frac{%
t^{p}}{p}\int_{\mathbb{R}^{3}}f(x)\left\vert u\right\vert ^{p}dx.
\end{equation*}%
By a calculation on the first and second derivatives, we find%
\begin{eqnarray*}
h_{u}^{\prime }\left( t\right)  &=&t\left\Vert u\right\Vert
_{H^{1}}^{2}+\lambda t^{3}\int_{\mathbb{R}^{3}}K(x)\phi
_{K,u}u^{2}dx-t^{p-1}\int_{\mathbb{R}^{3}}f(x)\left\vert u\right\vert ^{p}dx,
\\
h_{u}^{\prime \prime }\left( t\right)  &=&\left\Vert u\right\Vert
_{H^{1}}^{2}+3\lambda t^{2}\int_{\mathbb{R}^{3}}K(x)\phi
_{K,u}u^{2}dx-\left( p-1\right) t^{p-2}\int_{\mathbb{R}^{3}}f(x)\left\vert
u\right\vert ^{p}dx
\end{eqnarray*}%
and%
\begin{equation*}
th_{u}^{\prime }\left( t\right) =\left\Vert tu\right\Vert
_{H^{1}}^{2}+\lambda \int_{\mathbb{R}^{3}}K(x)\phi _{K,tu}\left( tu\right)
^{2}dx-\int_{\mathbb{R}^{3}}f(x)\left\vert tu\right\vert ^{p}dx.
\end{equation*}%
Thus, for any $u\in H^{1}(\mathbb{R}^{3})\backslash \left\{ 0\right\} $ and $%
t>0,$ $h_{u}^{\prime }\left( t\right) =0$ holds if and only if $tu\in
\mathbf{M}_{\lambda }$. In particular, $h_{u}^{\prime }\left( 1\right) =0$
if and only if $u\in \mathbf{M}_{\lambda }.$ It is natural to split $\mathbf{%
M}_{\lambda }$ into three parts corresponding to local minima, local maxima
and points of inflection. Accordingly, following \cite{T}, we define%
\begin{eqnarray*}
\mathbf{M}_{\lambda }^{+} &=&\{u\in \mathbf{M}_{\lambda }\text{ }|\text{ }%
h_{u}^{\prime \prime }\left( 1\right) >0\}; \\
\mathbf{M}_{\lambda }^{0} &=&\{u\in \mathbf{M}_{\lambda }\text{ }|\text{ }%
h_{u}^{\prime \prime }\left( 1\right) =0\}; \\
\mathbf{M}_{\lambda }^{-} &=&\{u\in \mathbf{M}_{\lambda }\text{ }|\text{ }%
h_{u}^{\prime \prime }\left( 1\right) <0\}.
\end{eqnarray*}

In order to look for nodal solutions of system $(SP)_{\lambda },$ we define
the nodal Nehari manifold by%
\begin{equation*}
\mathbf{N}_{\lambda }=\left\{ u\in H^{1}(\mathbb{R}^{3})\text{ }|\text{ }%
\left\langle I_{\lambda }^{\prime }\left( u\right) ,u^{+}\right\rangle
=\left\langle I_{\lambda }^{\prime }\left( u\right) ,u^{-}\right\rangle =0%
\text{ and }u^{\pm }\neq 0\right\} ,
\end{equation*}%
which is a subset of the Nehari manifold $\mathbf{M}_{\lambda }.$ Clearly, $%
u\in \mathbf{N}_{\lambda }$ if and only if%
\begin{equation*}
\left\langle I_{\lambda }^{\prime }\left( u\right) ,u^{+}\right\rangle
=\left\langle I_{\lambda }^{\prime }\left( u^{+}\right) ,u^{+}\right\rangle
+\lambda \int_{\mathbb{R}^{3}}K(x)\phi _{K,u^{-}}\left( u^{+}\right) ^{2}dx=0
\end{equation*}%
and
\begin{equation*}
\left\langle I_{\lambda }^{\prime }\left( u\right) ,u^{-}\right\rangle
=\left\langle I_{\lambda }^{\prime }\left( u^{-}\right) ,u^{-}\right\rangle
+\lambda \int_{\mathbb{R}^{3}}K(x)\phi _{K,u^{+}}\left( u^{-}\right)
^{2}dx=0.
\end{equation*}%
Moreover, by virtue of Lemma \ref{h1} $(ii),$ it is easy to verify that%
\begin{equation*}
\int_{\mathbb{R}^{3}}K(x)\phi _{K,u^{-}}\left( u^{+}\right) ^{2}dx=\int_{%
\mathbb{R}^{3}}K(x)\phi _{K,u^{+}}\left( u^{-}\right) ^{2}dx.
\end{equation*}%
For each $u\in \mathbf{M}_{\lambda },$ there holds%
\begin{eqnarray}
h_{u}^{\prime \prime }\left( 1\right)  &=&\left\Vert u\right\Vert
_{H^{1}}^{2}+3\lambda \int_{\mathbb{R}^{3}}K(x)\phi _{K,u}u^{2}dx-\left(
p-1\right) \int_{\mathbb{R}^{3}}f(x)\left\vert u\right\vert ^{p}dx  \notag \\
&=&-\left( p-2\right) \left\Vert u\right\Vert _{H^{1}}^{2}+\lambda \left(
4-p\right) \int_{\mathbb{R}^{3}}K(x)\phi _{K,u}u^{2}dx  \label{2-1} \\
&=&-2\left\Vert u\right\Vert _{H^{1}}^{2}+\left( 4-p\right) \int_{\mathbb{R}%
^{3}}f(x)\left\vert u\right\vert ^{p}dx  \label{2-2} \\
&\leq &-2\left\Vert u\right\Vert _{H^{1}}^{2}+\left( 4-p\right)
S_{p}^{-p}f_{\max }\Vert u\Vert _{H^{1}}^{p}.  \notag
\end{eqnarray}%
By $(\ref{2})$ and $(\ref{2-2}),$ we have
\begin{eqnarray*}
I_{\lambda }(u) &=&\frac{1}{4}\left\Vert u\right\Vert _{H^{1}}^{2}-\frac{4-p%
}{4p}\int_{\mathbb{R}^{3}}f(x)|u|^{p}dx \\
&>&\frac{p-2}{4p}\left\Vert u\right\Vert _{H^{1}}^{2} \\
&\geq &\frac{p-2}{4p}\left( \frac{S_{p}^{p}}{f_{\max }}\right) ^{\frac{2}{p-2%
}}>0\text{ for all }u\in \mathbf{M}_{\lambda }^{-},
\end{eqnarray*}%
which indicates that $I_{\lambda }$ is coercive and bounded below on $%
\mathbf{M}_{\lambda }^{-}.$

Let
\begin{equation*}
C\left( p\right) =\frac{A\left( p\right) \left( p-2\right) }{2p}\left( \frac{%
2}{4-p}\right) ^{\frac{2}{p-2}}\text{ for }2<p<4.
\end{equation*}%
It is not difficult to verify that $C\left( p\right) $ is increasing on $%
2<p<4$ and that%
\begin{equation*}
C\left( p\right) >\left\{
\begin{array}{ll}
\frac{\sqrt{e}\left( p-2\right) }{p}, & \text{ if }2<p\leq 3, \\
\frac{e\left( p-2\right) }{2p}, & \text{ if }3<p<4.%
\end{array}%
\right.
\end{equation*}%
Following \cite{SWF1}, for any $u\in \mathbf{M}_{\lambda }$ with $I_{\lambda
}\left( u\right) <C\left( p\right) \left( \frac{S_{p}^{p}}{f_{\infty }}%
\right) ^{\frac{2}{p-2}},$ we deduce that
\begin{eqnarray}
C\left( p\right) \left( \frac{S_{p}^{p}}{f_{\infty }}\right) ^{\frac{2}{p-2}%
} &>&I_{\lambda }(u)  \notag \\
&=&\frac{1}{2}\left\Vert u\right\Vert _{H^{1}}^{2}+\frac{\lambda }{4}\int_{%
\mathbb{R}^{3}}K(x)\phi _{K,u}u^{2}dx-\frac{1}{p}\int_{\mathbb{R}%
^{3}}f(x)|u|^{p}dx  \notag \\
&=&\frac{p-2}{2p}\left\Vert u\right\Vert _{H^{1}}^{2}-\frac{\lambda (4-p)}{4p%
}\int_{\mathbb{R}^{3}}K(x)\phi _{K,u}u^{2}dx  \label{2-4} \\
&\geq &\frac{p-2}{2p}\left\Vert u\right\Vert _{H^{1}}^{2}-\lambda \overline{S%
}^{-2}S_{12/5}^{-4}K_{\max }^{2}\left( \frac{4-p}{4p}\right) \Vert u\Vert
_{H^{1}}^{4}.  \label{2-5}
\end{eqnarray}%
It follows from (\ref{2-5}) that for $2<p<4$ and $0<\lambda <\lambda _{0},$
there exist two positive numbers $D_{1}$ and $D_{2}$ satisfying%
\begin{equation*}
\sqrt{A\left( p\right) }\left( \frac{2S_{p}^{p}}{f_{\infty }\left(
4-p\right) }\right) ^{\frac{1}{p-2}}<D_{1}<\left( \frac{2S_{p}^{p}}{f_{\max
}\left( 4-p\right) }\right) ^{\frac{1}{p-2}}<\sqrt{2}\left( \frac{2S_{p}^{p}%
}{f_{\infty }\left( 4-p\right) }\right) ^{\frac{1}{p-2}}<D_{2}
\end{equation*}%
such that%
\begin{equation*}
\left\Vert u\right\Vert _{H^{1}}<D_{1}\text{ or }\left\Vert u\right\Vert
_{H^{1}}>D_{2},
\end{equation*}%
where
\begin{equation*}
\lambda _{0}:=\frac{p-2}{2(4-p)}\left[ 1-A\left( p\right) \left( \frac{%
f_{\max }}{f_{\infty }}\right) ^{\frac{2}{p-2}}\right] \left( \frac{%
f_{\infty }(4-p)}{pS_{p}^{p}}\right) ^{\frac{2}{p-2}}\overline{S}%
^{2}S_{12/5}^{4}K_{\max }^{-2}>0.
\end{equation*}%
Note that $D_{1}\rightarrow \infty $ as $p\rightarrow 4^{-}.$ Thus, there
holds
\begin{eqnarray*}
\mathbf{M}_{\lambda }\left( C\left( p\right) \left( \frac{S_{p}^{p}}{%
f_{\infty }}\right) ^{\frac{2}{p-2}}\right) &=&\left\{ u\in \mathbf{M}%
_{\lambda }:J_{\lambda }\left( u\right) <C\left( p\right) \left( \frac{%
S_{p}^{p}}{f_{\infty }}\right) ^{\frac{2}{p-2}}\right\} \\
&=&\mathbf{M}_{\lambda }^{(1)}\cup \mathbf{M}_{\lambda }^{(2)},
\end{eqnarray*}%
where
\begin{equation*}
\mathbf{M}_{\lambda }^{(1)}:=\left\{ u\in \mathbf{M}_{\lambda }\left(
C\left( p\right) \left( \frac{S_{p}^{p}}{f_{\infty }}\right) ^{\frac{2}{p-2}%
}\right) :\left\Vert u\right\Vert _{H^{1}}<D_{1}\right\}
\end{equation*}%
and
\begin{equation*}
\mathbf{M}_{\lambda }^{(2)}:=\left\{ u\in \mathbf{M}_{\lambda }\left(
C\left( p\right) \left( \frac{S_{p}^{p}}{f_{\infty }}\right) ^{\frac{2}{p-2}%
}\right) :\left\Vert u\right\Vert _{H^{1}}>D_{2}\right\} .
\end{equation*}%
For $2<p<4$ and $0<\lambda <\lambda _{0},$ we further have
\begin{equation}
\left\Vert u\right\Vert _{H^{1}}<D_{1}<\left( \frac{2S_{p}^{p}}{f_{\max
}\left( 4-p\right) }\right) ^{\frac{1}{p-2}}\text{ for all }u\in \mathbf{M}%
_{\lambda }^{(1)}  \label{4-5}
\end{equation}%
and%
\begin{equation}
\left\Vert u\right\Vert _{H^{1}}>D_{2}>\sqrt{2}\left( \frac{2S_{p}^{p}}{%
f_{\infty }\left( 4-p\right) }\right) ^{\frac{1}{p-2}}\text{ for all }u\in
\mathbf{M}_{\lambda }^{(2)}.  \label{4-6}
\end{equation}%
From $(\ref{2-2}),(\ref{4-5})$ and the Sobolev inequality it follows that
\begin{equation*}
h_{\lambda ,u}^{\prime \prime }\left( 1\right) \leq -2\left\Vert
u\right\Vert _{H^{1}}^{2}+\left( 4-p\right) S_{p}^{-p}f_{\max }\left\Vert
u\right\Vert _{H^{1}}^{p}<0\text{ for all }u\in \mathbf{M}_{\lambda }^{(1)}.
\end{equation*}%
Using $\left( \ref{4-6}\right) $ leads to
\begin{eqnarray*}
\frac{1}{4}\left\Vert u\right\Vert _{H^{1}}^{2}-\frac{4-p}{4p}\int_{\mathbb{R%
}^{3}}f(x)|u|^{p}dx &=&J_{\lambda }\left( u\right) \\
&<&\frac{A\left( p\right) (p-2)}{2p}\left( \frac{2S_{p}^{p}}{f_{\infty
}\left( 4-p\right) }\right) ^{\frac{2}{p-2}} \\
&<&\frac{p-2}{2p}\left( \frac{2S_{p}^{p}}{f_{\infty }\left( 4-p\right) }%
\right) ^{\frac{2}{p-2}} \\
&<&\frac{p-2}{4p}\left\Vert u\right\Vert _{H^{1}}^{2}\text{ for all }u\in
\mathbf{M}_{\lambda }^{(2)}.
\end{eqnarray*}%
This implies that%
\begin{equation*}
2\left\Vert u\right\Vert _{H^{1}}^{2}<\left( 4-p\right) \int_{\mathbb{R}%
^{3}}f(x)|u|^{p}dx\text{ for all }u\in \mathbf{M}_{\lambda }^{(2)}.
\end{equation*}%
Combining the above inequality with $(\ref{2-2})$ gives%
\begin{equation*}
h_{\lambda ,u}^{\prime \prime }\left( 1\right) >0\text{ for all }u\in
\mathbf{M}_{\lambda }^{(2)}.
\end{equation*}

Set
\begin{equation*}
\mathbf{N}_{\lambda }^{\left( 1\right) }=\left\{ u\in \mathbf{M}_{\lambda
}^{(1)}\text{ }|\text{ }\left\langle I_{\lambda }^{\prime }\left( u\right)
,u^{+}\right\rangle =\left\langle I_{\lambda }^{\prime }\left( u\right)
,u^{-}\right\rangle =0\text{ and }u^{\pm }\not\equiv 0\right\} .
\label{2-15}
\end{equation*}%
Clearly, $\mathbf{N}_{\lambda }^{\left( 1\right) }$ is a subset of $\mathbf{M%
}_{\lambda }^{(1)},$ and also of $\mathbf{N}_{\lambda }.$ Moreover, for $%
u\in \mathbf{N}_{\lambda }^{\left( 1\right) },$ there holds%
\begin{equation}
\left( \frac{S_{p}^{p}}{f_{\max }}\right) ^{\frac{1}{p-2}}\leq \left\Vert
u^{\pm }\right\Vert _{H^{1}}<D_{1}<\left( \frac{2S_{p}^{p}}{\left(
4-p\right) f_{\max }}\right) ^{\frac{1}{p-2}}.  \label{2-20}
\end{equation}

Let%
\begin{eqnarray}
J_{\lambda }^{+}\left( u^{+},u^{-}\right)  &=&\frac{1}{2}\left\Vert
u^{+}\right\Vert _{H^{1}}^{2}+\frac{\lambda }{4}\left( \int_{\mathbb{R}%
^{3}}K(x)\phi _{K,u^{-}}(u^{+})^{2}dx+\int_{\mathbb{R}^{3}}K(x)\phi
_{K,u^{+}}(u^{+})^{2}dx\right)   \notag \\
&&-\frac{1}{p}\int_{\mathbb{R}^{3}}f(x)\left\vert u^{+}\right\vert ^{p}dx
\label{2-21}
\end{eqnarray}%
and%
\begin{eqnarray}
J_{\lambda }^{-}\left( u^{+},u^{-}\right)  &=&\frac{1}{2}\left\Vert
u^{-}\right\Vert _{H^{1}}^{2}+\frac{\lambda }{4}\left( \int_{\mathbb{R}%
^{3}}K(x)\phi _{K,u^{+}}(u^{-})^{2}dx+\int_{\mathbb{R}^{3}}K(x)\phi
_{K,u^{-}}(u^{-})^{2}dx\right)   \notag \\
&&-\frac{1}{p}\int_{\mathbb{R}^{3}}f(x)\left\vert u^{-}\right\vert ^{p}dx.
\label{2-25}
\end{eqnarray}%
Now we denote the function $\widetilde{h}\left( s,t\right) $ by%
\begin{equation}
\widetilde{h}\left( s,t\right) =J_{\lambda }^{+}\left( su^{+},tu^{-}\right)
+J_{\lambda }^{-}\left( su^{+},tu^{-}\right) \text{ for }s,t>0.  \label{2-28}
\end{equation}%
Clearly, $\widetilde{h}\left( s,t\right) =I_{\lambda }(su^{+}+tu^{-}).$
Moreover, a direct calculation shows that%
\begin{eqnarray*}
\frac{\partial }{\partial s}\widetilde{h}\left( s,t\right)  &=&s\left\Vert
u^{^{+}}\right\Vert _{H^{1}}^{2}+\lambda st^{2}\int_{\mathbb{R}^{3}}K(x)\phi
_{K,u^{-}}(u^{+})^{2}dx+\lambda s^{3}\int_{\mathbb{R}^{3}}K(x)\phi
_{K,u^{^{+}}}(u^{+})^{2}dx \\
&&-s^{p-1}\int_{\mathbb{R}^{3}}f(x)\left\vert u^{^{+}}\right\vert ^{p}dx, \\
\frac{\partial }{\partial t}\widetilde{h}\left( s,t\right)  &=&t\left\Vert
u^{^{-}}\right\Vert _{H^{1}}^{2}+\lambda s^{2}t\int_{\mathbb{R}^{3}}K(x)\phi
_{K,u^{+}}(u^{-})^{2}dx+\lambda t^{3}\int_{\mathbb{R}^{3}}K(x)\phi
_{K,u^{^{-}}}(u^{-})^{2}dx \\
&&-t^{p-1}\int_{\mathbb{R}^{3}}f(x)\left\vert u^{^{-}}\right\vert ^{p}dx
\end{eqnarray*}%
and%
\begin{eqnarray*}
\frac{\partial ^{2}}{\partial s^{2}}\widetilde{h}\left( s,t\right)
&=&\left\Vert u^{^{+}}\right\Vert _{H^{1}}^{2}+\lambda t^{2}\int_{\mathbb{R}%
^{3}}K(x)\phi _{K,u^{-}}(u^{+})^{2}dx+3\lambda s^{2}\int_{\mathbb{R}%
^{3}}K(x)\phi _{K,u^{^{+}}}(u^{+})^{2}dx \\
&&-\left( p-1\right) s^{p-2}\int_{\mathbb{R}^{3}}f(x)\left\vert
u^{^{+}}\right\vert ^{p}dx, \\
\frac{\partial ^{2}}{\partial t^{2}}\widetilde{h}\left( s,t\right)
&=&\left\Vert u^{^{-}}\right\Vert _{H^{1}}^{2}+\lambda s^{2}\int_{\mathbb{R}%
^{3}}K(x)\phi _{K,u^{+}}(u^{-})^{2}dx+3\lambda t^{2}\int_{\mathbb{R}%
^{3}}K(x)\phi _{K,u^{^{-}}}(u^{-})^{2}dx \\
&&-\left( p-1\right) t^{p-2}\int_{\mathbb{R}^{3}}f(x)\left\vert
u^{^{-}}\right\vert ^{p}dx.
\end{eqnarray*}%
If $u\in \mathbf{N}_{\lambda }^{\left( 1\right) },$ then $\frac{\partial }{%
\partial s}\widetilde{h}\left( 1,1\right) =\frac{\partial }{\partial t}%
\widetilde{h}\left( 1,1\right) =0,$%
\begin{eqnarray*}
\frac{\partial ^{2}}{\partial s^{2}}\widetilde{h}\left( 1,1\right)
&=&2\lambda \int_{\mathbb{R}^{3}}K(x)\phi _{K,u^{+}}(u^{+})^{2}dx-\left(
p-2\right) \int_{\mathbb{R}^{3}}f(x)\left\vert u^{+}\right\vert ^{p}dx \\
&=&-\left( p-2\right) \left( \left\Vert u^{+}\right\Vert
_{H^{1}}^{2}+\lambda \int_{\mathbb{R}^{3}}K(x)\phi
_{K,u^{-}}(u^{+})^{2}dx\right)  \\
&&+\left( 4-p\right) \lambda \int_{\mathbb{R}^{3}}K(x)\phi
_{K,u^{+}}(u^{+})^{2}dx \\
&=&-2\left( \left\Vert u^{+}\right\Vert _{H^{1}}^{2}+\lambda \int_{\mathbb{R}%
^{3}}K(x)\phi _{K,u^{-}}(u^{+})^{2}dx\right) +\left( 4-p\right) \int_{%
\mathbb{R}^{3}}f(x)\left\vert u^{+}\right\vert ^{p}dx \\
&\leq &-2\left\Vert u^{+}\right\Vert _{H^{1}}^{2}+\left( 4-p\right) f_{\max
}S_{p}^{-p}\Vert u^{+}\Vert _{H^{1}}^{p} \\
&=&\left( 4-p\right) \left\Vert u^{+}\right\Vert _{H^{1}}^{2}\left( f_{\max
}S_{p}^{-p}\Vert u^{+}\Vert _{H^{1}}^{p-2}-\frac{2}{4-p}\right) <0
\end{eqnarray*}%
and%
\begin{eqnarray*}
\frac{\partial ^{2}}{\partial t^{2}}\widetilde{h}\left( 1,1\right)
&=&2\lambda \int_{\mathbb{R}^{3}}K(x)\phi _{K,u^{-}}(u^{-})^{2}dx-\left(
p-2\right) \int_{\mathbb{R}^{3}}f(x)\left\vert u^{-}\right\vert ^{p}dx \\
&=&-\left( p-2\right) \left( \left\Vert u^{-}\right\Vert
_{H^{1}}^{2}+\lambda \int_{\mathbb{R}^{3}}K(x)\phi
_{K,u^{+}}(u^{-})^{2}dx\right)  \\
&&+\left( 4-p\right) \lambda \int_{\mathbb{R}^{3}}K(x)\phi
_{K,u^{-}}(u^{-})^{2}dx \\
&=&-2\left( \left\Vert u^{-}\right\Vert _{H^{1}}^{2}+\lambda \int_{\mathbb{R}%
^{3}}K(x)\phi _{K,u^{+}}(u^{-})^{2}dx\right) +\left( 4-p\right) \int_{%
\mathbb{R}^{3}}f(x)\left\vert u^{-}\right\vert ^{p}dx \\
&\leq &-2\left\Vert u^{-}\right\Vert _{H^{1}}^{2}+\left( 4-p\right) f_{\max
}S_{p}^{-p}\Vert u^{-}\Vert _{H^{1}}^{p} \\
&=&\left( 4-p\right) \left\Vert u^{-}\right\Vert _{H^{1}}^{2}\left( f_{\max
}S_{p}^{-p}\Vert u^{-}\Vert _{H^{1}}^{p-2}-\frac{2}{4-p}\right) <0.
\end{eqnarray*}%
Furthermore, we have the following result.

\begin{lemma}
\label{h3-3}Suppose that $2<p<4$ and conditions $\left( F1\right) $ and $%
\left( K1\right) $ hold. Then there exists a positive number $\widetilde{%
\lambda }\leq \lambda _{0}$ such that for every $0<\lambda <\widetilde{%
\lambda }$ and $u\in \mathbf{N}_{\lambda }^{\left( 1\right) },$ there exist $%
\left( \frac{p}{2}\right) ^{\frac{1}{p-2}}<\widetilde{s}_{\lambda },%
\widetilde{t}_{\lambda }\leq \left( \frac{p}{4-p}\right) ^{\frac{1}{p-2}}$
such that $I_{\lambda }\left( \widetilde{s}_{\lambda }u^{+}+\widetilde{t}%
_{\lambda }u^{-}\right) <0.$ Furthermore, there holds%
\begin{equation*}
I_{\lambda }\left( u^{+}+u^{-}\right) =\sup_{\left( s,t\right) \in \left[ 0,%
\widetilde{s}_{\lambda }\right] \times \left[ 0,\widetilde{t}_{\lambda }%
\right] }I_{\lambda }(su^{+}+tu^{-}).
\end{equation*}
\end{lemma}

\begin{proof}
By Lemma \ref{h1} $(i)$ and Young's inequality,%
\begin{equation*}
s^{2}t^{2}\int_{\mathbb{R}^{3}}K(x)\phi _{K,u^{-}}(u^{+})^{2}dx\leq \frac{%
s^{4}}{2}\overline{S}^{-2}S_{12/5}^{-4}K_{\max }^{2}\left\Vert
u^{+}\right\Vert _{H^{1}}^{4}+\frac{t^{4}}{2}\overline{S}%
^{-2}S_{12/5}^{-4}K_{\max }^{2}\left\Vert u^{-}\right\Vert _{H^{1}}^{4}.
\end{equation*}%
Using the above inequality, together with $(\ref{2-21})-(\ref{2-28})$ leads
to%
\begin{eqnarray*}
I_{\lambda }\left( su^{+}+tu^{-}\right)  &=&\widetilde{h}\left( s,t\right)
\\
&\leq &\frac{s^{2}}{2}\left\Vert u^{+}\right\Vert _{H^{1}}^{2}+\frac{\lambda
s^{4}}{4}\int_{\mathbb{R}^{3}}K(x)\phi _{K,u^{+}}(u^{+})^{2}dx-\frac{s^{p}}{p%
}\int_{\mathbb{R}^{3}}f(x)\left\vert u^{+}\right\vert ^{p}dx \\
&&+\frac{t^{2}}{2}\left\Vert u^{-}\right\Vert _{H^{1}}^{2}+\frac{\lambda
t^{4}}{4}\int_{\mathbb{R}^{3}}K(x)\phi _{K,u^{-}}(u^{-})^{2}dx-\frac{t^{p}}{p%
}\int_{\mathbb{R}^{3}}f(x)\left\vert u^{-}\right\vert ^{p}dx \\
&&+\frac{\lambda s^{4}}{4}\overline{S}^{-2}S_{12/5}^{-4}K_{\max
}^{2}\left\Vert u^{+}\right\Vert _{H^{1}}^{4}+\frac{\lambda t^{4}}{4}%
\overline{S}^{-2}S_{12/5}^{-4}K_{\max }^{2}\left\Vert u^{-}\right\Vert
_{H^{1}}^{4} \\
&\leq &g^{+}(s)+g^{-}(t),
\end{eqnarray*}%
where%
\begin{equation*}
g^{+}\left( s\right) =\frac{s^{2}}{2}\left\Vert u^{+}\right\Vert
_{H^{1}}^{2}+\frac{\lambda s^{4}}{2}\overline{S}^{-2}S_{12/5}^{-4}K_{\max
}^{2}\left\Vert u^{+}\right\Vert _{H^{1}}^{4}-\frac{s^{p}}{p}\int_{\mathbb{R}%
^{3}}f(x)\left\vert u^{+}\right\vert ^{p}dx
\end{equation*}%
and%
\begin{equation*}
g^{-}\left( t\right) =\frac{t^{2}}{2}\left\Vert u^{-}\right\Vert
_{H^{1}}^{2}+\frac{\lambda t^{4}}{2}\overline{S}^{-2}S_{12/5}^{-4}K_{\max
}^{2}\left\Vert u^{-}\right\Vert _{H^{1}}^{4}-\frac{t^{p}}{p}\int_{\mathbb{R}%
^{3}}f(x)\left\vert u^{-}\right\vert ^{p}dx.
\end{equation*}%
In order to arrive at the conclusion, we only need to show that there exist $%
\widetilde{s}_{\lambda },\widetilde{t}_{\lambda }>0$ such that $g^{+}\left(
\widetilde{s}_{\lambda }\right) ,g^{-}\left( \widetilde{t}_{\lambda }\right)
<0.$

Let
\begin{equation*}
\widetilde{g}(t)=\frac{t^{-2}}{2}\left\Vert u^{-}\right\Vert _{H^{1}}^{2}-%
\frac{t^{p-4}}{p}\int_{\mathbb{R}^{3}}f(x)\left\vert u^{-}\right\vert ^{p}dx%
\text{ for }t>0.
\end{equation*}%
A straightforward calculation gives
\begin{equation*}
\widetilde{g}(t_{\lambda })=0,\ \lim_{t\rightarrow 0^{+}}\widetilde{g}%
(t)=\infty \ \text{and}\ \lim_{t\rightarrow \infty }\widetilde{g}(t)=0,
\end{equation*}%
where
\begin{equation*}
t_{\lambda }:=\left( \frac{p\left\Vert u^{-}\right\Vert _{H^{1}}^{2}}{2\int_{%
\mathbb{R}^{3}}f(x)\left\vert u^{-}\right\vert ^{p}dx}\right) ^{\frac{1}{p-2}%
}.
\end{equation*}%
By the fact of $\frac{\partial }{\partial t}\widetilde{h}\left( 1,1\right) =0
$ and $(\ref{2-20})$ one has%
\begin{equation}
\left( \frac{p(4-p)}{4}\right) ^{\frac{1}{p-2}}<t_{\lambda }\leq \left(
\frac{p}{2}\right) ^{\frac{1}{p-2}}.  \label{2-17}
\end{equation}%
By calculating the derivative of $\widetilde{g}(t),$ we find%
\begin{eqnarray*}
\widetilde{g}^{\prime }\left( t\right)  &=&-t^{-3}\left\Vert
u^{-}\right\Vert _{H^{1}}^{2}+\frac{\left( 4-p\right) t^{p-5}}{p}\int_{%
\mathbb{R}^{3}}f(x)\left\vert u^{-}\right\vert ^{p}dx \\
&=&t^{-3}\left[ -\left\Vert u^{-}\right\Vert _{H^{1}}^{2}+\frac{\left(
4-p\right) t^{p-2}}{p}\int_{\mathbb{R}^{3}}f(x)\left\vert u^{-}\right\vert
^{p}dx\right] ,
\end{eqnarray*}%
which indicates that there exists $\widetilde{t}_{\lambda }=\left( \frac{2}{%
4-p}\right) ^{\frac{1}{p-2}}t_{\lambda }$ such that $\widetilde{g}\left(
t\right) $ is decreasing when $0<t<\widetilde{t}_{\lambda }$ and is
increasing when $t>\widetilde{t}_{\lambda }.$ Moreover, using $\left( \ref%
{2-17}\right) $ gives
\begin{equation}
1<\left( \frac{p}{2}\right) ^{\frac{1}{p-2}}<\widetilde{t}_{\lambda }\leq
\left( \frac{p}{4-p}\right) ^{\frac{1}{p-2}}.  \label{2-18}
\end{equation}%
Thus, by virtue of $\left( \ref{2-20}\right) $ and $(\ref{2-18}),$ we have%
\begin{eqnarray}
\inf_{t>0}\widetilde{g}\left( t\right)  &=&\widetilde{g}\left( \widetilde{t}%
_{\lambda }\right)   \notag \\
&=&-\frac{p-2}{2(4-p)}\left[ \frac{p\left\Vert u^{-}\right\Vert _{H^{1}}^{2}%
}{(4-p)\int_{\mathbb{R}^{3}}f(x)\left\vert u^{-}\right\vert ^{p}dx}\right]
^{-\frac{2}{p-2}}\left\Vert u^{-}\right\Vert _{H^{1}}^{2}  \notag \\
&\leq &-\frac{p-2}{2(4-p)}\left( \frac{p}{4-p}\right) ^{-\frac{2}{p-2}%
}\left\Vert u^{-}\right\Vert _{H^{1}}^{2}  \notag \\
&\leq &-\frac{p-2}{2(4-p)}\left( \frac{S_{p}^{p}(4-p)}{pf_{\max }}\right) ^{%
\frac{2}{p-2}},  \notag
\end{eqnarray}%
which implies that there exists a positive constant $\lambda _{1}\leq
\lambda _{0}$ such that for every $0<\lambda <\lambda _{1},$%
\begin{equation*}
\inf_{t>0}\widetilde{g}\left( t\right) =\widetilde{g}\left( \widetilde{t}%
_{\lambda }\right) <-\frac{\lambda }{2}\overline{S}^{-2}S_{12/5}^{-4}K_{\max
}^{2}\left\Vert u^{-}\right\Vert _{H^{1}}^{4}.
\end{equation*}%
This indicates that
\begin{eqnarray}
g^{-}\left( \widetilde{t}_{\lambda }\right)  &=&\frac{\widetilde{t}_{\lambda
}^{2}}{2}\left\Vert u^{-}\right\Vert _{H^{1}}^{2}+\frac{\lambda \widetilde{t}%
_{\lambda }^{4}}{2}\overline{S}^{-2}S_{12/5}^{-4}K_{\max }^{2}\left\Vert
u^{-}\right\Vert _{H^{1}}^{4}-\frac{\widetilde{t}_{\lambda }^{p}}{p}\int_{%
\mathbb{R}^{3}}f(x)\left\vert u^{-}\right\vert ^{p}dx  \notag \\
&=&\widetilde{t}_{\lambda }^{4}\left( \widetilde{g}\left( \widetilde{t}%
_{\lambda }\right) +\frac{\lambda }{2}\overline{S}^{-2}S_{12/5}^{-4}K_{\max
}^{2}\left\Vert u^{-}\right\Vert _{H^{1}}^{4}\right) <0.  \label{2-12}
\end{eqnarray}%
Similarly, we also obtain that there exists%
\begin{equation*}
1<\left( \frac{p}{2}\right) ^{\frac{1}{p-2}}<\widetilde{s}_{\lambda }\leq
\left( \frac{p}{4-p}\right) ^{\frac{1}{p-2}}
\end{equation*}%
such that
\begin{equation}
g^{+}\left( \widetilde{s}_{\lambda }\right) =\frac{\widetilde{s}_{\lambda
}^{2}}{2}\left\Vert u^{+}\right\Vert _{H^{1}}^{2}+\frac{\lambda \widetilde{s}%
_{\lambda }^{4}}{2}\overline{S}^{-2}S_{12/5}^{-4}K_{\max }^{2}\left\Vert
u^{+}\right\Vert _{H^{1}}^{4}-\frac{\widetilde{s}_{\lambda }^{p}}{p}\int_{%
\mathbb{R}^{3}}f(x)\left\vert u^{+}\right\vert ^{p}dx<0.  \label{2-24}
\end{equation}%
Thus, by $(\ref{2-12})$ and $(\ref{2-24}),$ for every $u\in \mathbf{N}%
_{\lambda }^{\left( 1\right) }$ there holds%
\begin{equation*}
I_{\lambda }\left( \widetilde{s}_{\lambda }u^{+}+\widetilde{t}_{\lambda
}u^{-}\right) =\widetilde{h}\left( \widetilde{s}_{\lambda },\widetilde{t}%
_{\lambda }\right) <0\text{ for all }0<\lambda <\lambda _{1}.
\end{equation*}%
Next, we show that%
\begin{equation*}
\sup_{\left( s,t\right) \in \left[ 0,\widetilde{s}_{\lambda }\right] \times %
\left[ 0,\widetilde{t}_{\lambda }\right] }I_{\lambda }\left(
su^{+}+tu^{-}\right) =I_{\lambda }(u^{+}+u^{-}).
\end{equation*}%
Set $Q_{\lambda }=\left[ 0,\widetilde{s}_{\lambda }\right] \times \left[ 0,%
\widetilde{t}_{\lambda }\right] .$ First, we claim that
\begin{equation*}
\sup_{\left( s,t\right) \in \partial Q_{\lambda }}I_{\lambda }\left(
su^{+}+tu^{-}\right) <\sup_{\left( s,t\right) \in Q_{\lambda }}I_{\lambda
}(su^{+}+tu^{-}).
\end{equation*}%
Let us define%
\begin{eqnarray*}
A_{1} &=&\left\Vert u^{+}\right\Vert _{H^{1}}^{2},A_{2}=\int_{\mathbb{R}%
^{3}}K(x)\phi _{K,u^{+}}(u^{+})^{2}dx,A_{3}=\int_{\mathbb{R}%
^{3}}f(x)\left\vert u^{+}\right\vert ^{p}dx, \\
B_{1} &=&\left\Vert u^{-}\right\Vert _{H^{1}}^{2},B_{2}=\int_{\mathbb{R}%
^{3}}K(x)\phi _{K,u^{-}}(u^{-})^{2}dx,B_{3}=\int_{\mathbb{R}%
^{3}}f(x)\left\vert u^{-}\right\vert ^{p}dx
\end{eqnarray*}%
and%
\begin{equation*}
C=\int_{\mathbb{R}^{3}}K(x)\phi _{K,u^{-}}(u^{+})^{2}dx=\int_{\mathbb{R}%
^{3}}K(x)\phi _{K,u^{+}}(u^{-})^{2}dx.
\end{equation*}%
Then%
\begin{equation*}
\widetilde{h}\left( s,t\right) =\frac{s^{2}}{2}A_{1}+\frac{t^{2}}{2}%
B_{1}+\lambda \frac{s^{4}}{4}A_{2}+\lambda \frac{t^{4}}{4}B_{2}+\lambda
\frac{s^{2}t^{2}}{2}C-\frac{s^{p}}{p}A_{3}-\frac{t^{p}}{p}B_{3}.
\end{equation*}%
Clearly, there holds%
\begin{eqnarray*}
\frac{\partial }{\partial s}\widetilde{h}\left( s,t\right)  &=&s\left(
A_{1}+\lambda s^{2}A_{2}+\lambda t^{2}C-s^{p-2}A_{3}\right) ; \\
\frac{\partial }{\partial t}\widetilde{h}\left( s,t\right)  &=&t\left(
B_{1}+\lambda t^{2}B_{2}+\lambda s^{2}C-t^{p-2}B_{3}\right) .
\end{eqnarray*}%
It is not difficult to obtain that there exist $s_{0},t_{0}>0$ sufficiently
small such that%
\begin{equation}
\frac{\partial }{\partial s}\widetilde{h}\left( s,t\right) >0\text{ for all }%
\left( s,t\right) \in \left( 0,s_{0}\right) \times \left[ 0,\widetilde{t}%
_{\lambda }\right]   \label{3-2}
\end{equation}%
and
\begin{equation}
\frac{\partial }{\partial t}\widetilde{h}\left( s,t\right) >0\text{ for all }%
\left( s,t\right) \in \left[ 0,\widetilde{s}_{\lambda }\right] \times
(0,t_{0}).  \label{3-3}
\end{equation}%
Note that $A_{1}<A_{3},B_{1}<B_{3}$ and $\widetilde{s}_{\lambda },\widetilde{%
t}_{\lambda }>\left( \frac{p}{2}\right) ^{\frac{1}{p-2}}>1.$ Then there
exists a positive constant $\lambda _{2}\leq \lambda _{1}$ such that for
every $0<\lambda <\lambda _{2},$%
\begin{equation}
\frac{\partial }{\partial s}\widetilde{h}\left( s,t\right) <0\text{ for all }%
\left( s,t\right) \in \left\{ \widetilde{s}_{\lambda }\right\} \times \left[
0,\widetilde{t}_{\lambda }\right]   \label{3-4}
\end{equation}%
and
\begin{equation}
\frac{\partial }{\partial t}\widetilde{h}\left( s,t\right) <0\text{ for all }%
\left( s,t\right) \in \left[ 0,\widetilde{s}_{\lambda }\right] \times
\left\{ \widetilde{t}_{\lambda }\right\} .  \label{3-5}
\end{equation}%
By $(\ref{3-2})-(\ref{3-5}),$ we can conclude that%
\begin{equation*}
\sup_{\left( s,t\right) \in \partial Q_{\lambda }}I_{\lambda
}(su^{+}+tu^{-})<\sup_{\left( s,t\right) \in Q_{\lambda }}I_{\lambda
}(su^{+}+tu^{-}).
\end{equation*}%
Second, we prove that $I_{\lambda }\left( u^{+}+u^{-}\right) =\sup_{\left(
s,t\right) \in Q_{\lambda }}I_{\lambda }(su^{+}+tu^{-}).$ Since $\frac{%
\partial }{\partial s}\widetilde{h}\left( 1,1\right) =\frac{\partial }{%
\partial t}\widetilde{h}\left( 1,1\right) =0,$ we have $\left( 1,1\right) $
is a critical point of $\widetilde{h}\left( s,t\right) $ for all $\lambda >0.
$ By a calculation, we deduce that%
\begin{eqnarray*}
\frac{\partial ^{2}}{\partial s^{2}}\widetilde{h}\left( s,t\right)
&=&A_{1}+3\lambda s^{2}A_{2}+\lambda t^{2}C-\left( p-1\right) s^{p-2}A_{3};
\\
\frac{\partial ^{2}}{\partial t^{2}}\widetilde{h}\left( s,t\right)
&=&B_{1}+3\lambda t^{2}B_{2}+\lambda s^{2}C-\left( p-1\right) t^{p-2}B_{3};
\\
\frac{\partial ^{2}}{\partial s\partial t}\widetilde{h}\left( s,t\right)
&=&2\lambda stC.
\end{eqnarray*}%
Then the Hessian matric of $\widetilde{h}$ at $\left( 1,1\right) $ is%
\begin{eqnarray*}
H_{\lambda } &=&\left[
\begin{array}{cc}
A_{1}+3\lambda A_{2}+\lambda C-\left( p-1\right) A_{3} & 2\lambda C \\
2\lambda C & B_{1}+3\lambda B_{2}+\lambda C-\left( p-1\right) B_{3}%
\end{array}%
\right]  \\
&=&\left[
\begin{array}{cc}
A_{1}-\left( p-1\right) A_{3} & 0 \\
0 & B_{1}-\left( p-1\right) B_{3}%
\end{array}%
\right] +\lambda \left[
\begin{array}{cc}
3A_{2}+C & 2C \\
2C & 3B_{2}+C%
\end{array}%
\right]
\end{eqnarray*}%
for all $\lambda >0.$ We notice that the matrix%
\begin{equation*}
-\left[
\begin{array}{cc}
A_{1}-\left( p-1\right) A_{3} & 0 \\
0 & B_{1}-\left( p-1\right) B_{3}%
\end{array}%
\right]
\end{equation*}%
is positive definite, since $0<A_{1}<A_{3},0<B_{1}<B_{3}$ and $2<p<4.$ Using
this, together with the fact that $A_{2},B_{2},C$ are uniformly bounded for
all $\lambda >0$, we get $-H_{\lambda }$ is positive definite for $\lambda >0
$ sufficiently small. This implies that there exists $r_{0}>0$ sufficiently
small, independent of $\lambda $ such that $\widetilde{h}\left( 1,1\right) $
is a unique global maximum point on
\begin{equation*}
B_{r_{0}}\left( \left( 1,1\right) \right) =\left\{ \left( s,t\right) :s,t>0%
\text{ and }\left\vert \left( s,t\right) -\left( 1,1\right) \right\vert
<r_{0}\right\} \subset Q_{\lambda }.
\end{equation*}%
Next, we show that $\widetilde{h}\left( 1,1\right) $ is a unique global
maximum on $Q_{\lambda }$ for $\lambda >0$ sufficiently small$.$ If not,
there exist a sequence $\left\{ \lambda _{n}\right\} \subset \mathbb{R}^{+}$
with $\lambda _{n}\rightarrow 0$ as $n\rightarrow \infty $ and points $%
\left( s_{\lambda _{n}},t_{\lambda _{n}}\right) \in Q_{\lambda
_{n}}\backslash B_{r_{0}}\left( \left( 1,1\right) \right) $ such that
\begin{equation*}
\widetilde{h}\left( s_{\lambda _{n}},t_{\lambda _{n}}\right) =\sup_{\left(
s,t\right) \in Q_{\lambda _{n}}\backslash B_{r_{0}}\left( \left( 1,1\right)
\right) }\widetilde{h}(s,t).
\end{equation*}%
Since
\begin{equation*}
Q_{\lambda _{n}}\subset \left[ 0,\left( \frac{p}{4-p}\right) ^{\frac{1}{p-2}}%
\right] \times \left[ 0,\left( \frac{p}{4-p}\right) ^{\frac{1}{p-2}}\right] ,
\end{equation*}%
we have $\left\{ \left( s_{\lambda _{n}},t_{\lambda _{n}}\right) \right\} $
is a bounded sequence. Then there exist a subsequence $\left\{ \left(
s_{\lambda _{n}},t_{\lambda _{n}}\right) \right\} $ and
\begin{equation*}
\left( s_{0},t_{0}\right) \in \left[ 0,\left( \frac{p}{4-p}\right) ^{\frac{1%
}{p-2}}\right] \times \left[ 0,\left( \frac{p}{4-p}\right) ^{\frac{1}{p-2}}%
\right] \backslash B_{r_{0}}\left( \left( 1,1\right) \right)
\end{equation*}%
such that
\begin{equation*}
\left( s_{\lambda _{n}},t_{\lambda _{n}}\right) \rightarrow \left(
s_{0},t_{0}\right) \text{ as }n\rightarrow \infty
\end{equation*}%
and%
\begin{equation*}
\widetilde{h}\left( s_{0},t_{0}\right) \geq \widetilde{h}(1,1),
\end{equation*}%
which contradicts to the fact that $\left( 1,1\right) $ is a unique global
maximum point of $\widetilde{h}$ for $\lambda =0.$ Therefore, there exists a
positive constant $\widetilde{\lambda }\leq \lambda _{2}$ such that for
every $0<\lambda <\widetilde{\lambda },$
\begin{equation*}
I_{\lambda }\left( u^{+}+u^{-}\right) =\sup_{\left( s,t\right) \in \left[ 0,%
\widetilde{s}_{\lambda }\right] \times \left[ 0,\widetilde{t}_{\lambda }%
\right] }I_{\lambda }\left( su^{+}+tu^{-}\right) .
\end{equation*}%
This completes the proof.
\end{proof}

Similar to the argument in Lemma \ref{h3-3}, we obtain that for every $%
0<\lambda <\widetilde{\lambda }$ and $u\in \mathbf{N}_{\lambda }^{\left(
1\right) }$ there exist $\left( \frac{p}{2}\right) ^{\frac{1}{p-2}%
}<s_{\lambda }^{+},t_{\lambda }^{-}\leq \left( \frac{p}{4-p}\right) ^{\frac{1%
}{p-2}}$ such that $J_{\lambda }^{+}\left( s_{\lambda
}^{+}u^{+},u^{-}\right) <0,J_{\lambda }^{-}\left( u^{+},t_{\lambda
}^{-}u^{-}\right) <0$ and%
\begin{equation*}
J_{\lambda }^{+}\left( u^{+},u^{-}\right) =\sup_{s\in \left[ 0,s_{\lambda
}^{+}\right] }J_{\lambda }^{+}\left( su^{+},u^{-}\right) ;\text{ }J_{\lambda
}^{-}\left( u^{+},u^{-}\right) =\sup_{t\in \left[ 0,t_{\lambda }^{-}\right]
}J_{\lambda }^{-}(u^{+},tu^{-}).
\end{equation*}%
Furthermore, we have the following result.

\begin{lemma}
\label{h3-5}Suppose that $2<p<4,$ and conditions $\left( F1\right) $ and $%
\left( K1\right) $ hold. Let $\widetilde{\lambda }>0$ be as in Lemma \ref%
{h3-3}. Then for every $0<\lambda <\widetilde{\lambda }$ and $u\in \mathbf{N}%
_{\lambda }^{\left( 1\right) }$ there exist $0<s_{0}^{+},t_{0}^{-}\leq 1$
such that $s_{0}^{+}u^{+},t_{0}^{-}u^{-}\in \mathbf{M}_{\lambda }^{-}$ and%
\begin{equation*}
\sup_{s\in \left[ 0,s_{\lambda }^{+}\right] }I_{\lambda }\left(
su^{+}\right) =I_{\lambda }\left( s_{0}^{+}u^{+}\right) \geq \alpha
_{\lambda }^{-};\text{ }\sup_{t\in \left[ 0,t_{\lambda }^{-}\right]
}I_{\lambda }\left( tu^{-}\right) =I_{\lambda }\left( t_{0}^{-}u^{-}\right)
\geq \alpha _{\lambda }^{-},
\end{equation*}%
where $\alpha _{\lambda }^{-}=\inf_{u\in \mathbf{M}_{\lambda
}^{-}}I_{\lambda }(u).$ In particular,
\begin{equation*}
J_{\lambda }^{\pm }\left( u^{+},u^{-}\right) \geq \alpha _{\lambda }^{-}%
\text{ for all }u\in \mathbf{N}_{\lambda }^{\left( 1\right) }.
\end{equation*}
\end{lemma}

\begin{proof}
We only prove the case of $"+",$ since the case of $"-"$ is analogous. Let
\begin{eqnarray*}
\widehat{h}_{u^{+}}\left( s\right) &=&I_{\lambda }\left( su^{+}\right) \\
&=&\frac{1}{2}\left\Vert su^{+}\right\Vert _{H^{1}}^{2}+\frac{\lambda }{4}%
\int_{\mathbb{R}^{3}}K(x)\phi _{K,su^{+}}(su^{+})^{2}dx-\frac{1}{p}\int_{%
\mathbb{R}^{3}}f(x)|su^{+}|^{p}dx
\end{eqnarray*}%
for $s>0.$ Clearly,%
\begin{eqnarray*}
\widehat{h}_{u^{+}}^{\prime }\left( s\right) &=&s\left\Vert u^{+}\right\Vert
_{H^{1}}^{2}+\lambda s^{3}\int_{\mathbb{R}^{3}}K(x)\phi
_{K,u^{+}}(u^{+})^{2}dx-s^{p-1}\int_{\mathbb{R}^{3}}f(x)|u^{+}|^{p}dx \\
&=&s^{3}\left( \widehat{g}_{u^{+}}\left( s\right) +\lambda \int_{\mathbb{R}%
^{3}}K(x)\phi _{K,u^{+}}(u^{+})^{2}dx\right) ,
\end{eqnarray*}%
where%
\begin{equation*}
\widehat{g}_{u^{+}}\left( s\right) =s^{-2}\left\Vert u^{+}\right\Vert
_{H^{1}}^{2}-s^{p-4}\int_{\mathbb{R}^{3}}f(x)|u^{+}|^{p}dx.
\end{equation*}%
Analyzing the functions $\widehat{g}_{u^{+}}$ leads to
\begin{equation*}
\widehat{g}_{u^{+}}\left( \widehat{s}\right) =0,\ \lim_{s\rightarrow 0^{+}}%
\widehat{g}_{u^{+}}(s)=\infty \ \text{and}\ \lim_{s\rightarrow \infty }%
\widehat{g}_{u^{+}}(s)=0,
\end{equation*}%
where%
\begin{equation*}
\left( \frac{4-p}{2}\right) ^{\frac{1}{p-2}}<\widehat{s}:=\left( \frac{%
\left\Vert u^{+}\right\Vert _{H^{1}}^{2}}{\int_{\mathbb{R}%
^{3}}f(x)|u^{+}|^{p}dx}\right) ^{\frac{1}{p-2}}\leq 1.
\end{equation*}%
Moreover, the derivative of $\widehat{g}_{u^{+}}\left( s\right) \ $is the
following%
\begin{eqnarray*}
\widehat{g}_{u^{+}}^{\prime }\left( s\right) &=&-2s^{-3}\left\Vert
u^{+}\right\Vert _{H^{1}}^{2}+\left( 4-p\right) s^{p-5}\int_{\mathbb{R}%
^{3}}f(x)|u^{+}|^{p}dx \\
&=&s^{-3}\left( -2\left\Vert u^{+}\right\Vert _{H^{1}}^{2}+\left( 4-p\right)
s^{p-2}\int_{\mathbb{R}^{3}}f(x)|u^{+}|^{p}dx\right) ,
\end{eqnarray*}%
which indicates that $\widehat{g}_{u^{+}}\left( s\right) $ is decreasing
when $0<s<\left( \frac{2}{4-p}\right) ^{\frac{1}{p-2}}\widehat{s}$ and is
increasing when $s>\left( \frac{2}{4-p}\right) ^{\frac{1}{p-2}}\widehat{s}$
and
\begin{eqnarray*}
\inf_{s>0}\widehat{g}_{u^{+}}\left( s\right) &=&\widehat{g}_{u^{+}}\left(
\left( \frac{2}{4-p}\right) ^{\frac{1}{p-2}}\widehat{s}\right) \\
&=&\frac{2-p}{4-p}\left( \frac{\left\Vert u^{+}\right\Vert _{H^{1}}^{2}}{%
\int_{\mathbb{R}^{3}}f(x)\left\vert u^{+}\right\vert ^{p}dx}\right) ^{-\frac{%
2}{p-2}}\left( \frac{2}{4-p}\right) ^{-\frac{2}{p-2}}\left\Vert
u^{+}\right\Vert _{H^{1}}^{2} \\
&=&-\frac{p-2}{4-p}\left( \frac{4-p}{2}\right) ^{\frac{2}{p-2}}\widehat{s}%
^{-2}\Vert u^{+}\Vert _{H^{1}}^{2} \\
&<&-\frac{p-2}{4-p}\left( \frac{4-p}{2}\right) ^{\frac{2}{p-2}}\Vert
u^{+}\Vert _{H^{1}}^{2}.
\end{eqnarray*}%
Note that%
\begin{equation*}
\left( \frac{S_{p}^{p}}{f_{\max }}\right) ^{\frac{1}{p-2}}\leq \left\Vert
u^{+}\right\Vert _{H^{1}}<\left( \frac{2S_{p}^{p}}{f_{\max }\left(
4-p\right) }\right) ^{\frac{1}{p-2}}.
\end{equation*}%
Similar to the argument in Lemma \ref{h3-3}, we obtain that for $0<\lambda <%
\widetilde{\lambda },$%
\begin{equation*}
\inf_{s>0}\widehat{g}_{u^{+}}\left( s\right) <-\lambda \int_{\mathbb{R}%
^{3}}K(x)\phi _{K,u^{+}}(u^{+})^{2}dx.
\end{equation*}%
Then there are two numbers $s_{0}^{+}$ and $\widehat{s}_{0}^{+}$ satisfying%
\begin{equation*}
\widehat{s}<s_{0}^{+}<\left( \frac{2}{4-p}\right) ^{\frac{1}{p-2}}\widehat{s}%
<\widehat{s}_{0}^{+}
\end{equation*}%
such that
\begin{equation*}
\widehat{g}_{u^{+}}\left( s_{0}^{+}\right) +\lambda \int_{\mathbb{R}%
^{3}}K(x)\phi _{K,u^{+}}(u^{+})^{2}dx=0
\end{equation*}%
and
\begin{equation*}
\widehat{g}_{u^{+}}\left( \widehat{s}_{0}^{+}\right) +\lambda \int_{\mathbb{R%
}^{3}}K(x)\phi _{K,u^{+}}(u^{+})^{2}dx=0.
\end{equation*}%
Moreover, $\widehat{h}_{u^{+}}\left( s\right) $ is increasing when $s\in
\left( 0,s_{0}^{+}\right) \cup \left( \widehat{s}_{0}^{+},\infty \right) $
and is decreasing when $s_{0}^{+}<s<\widehat{s}_{0}^{+}.$ Note that $%
\widehat{h}_{u^{+}}\left( s_{\lambda }^{+}\right) =I_{\lambda }\left(
s_{\lambda }^{+}u^{+}\right) <0$ by the fact of $J_{\lambda }^{+}\left(
s_{\lambda }^{+}u^{+},u^{-}\right) <0.$ Thus, there holds $s_{0}^{+}u^{+}\in
\mathbf{M}_{\lambda }^{-}$ and%
\begin{equation*}
\sup_{s\in \left[ 0,s_{\lambda }^{+}\right] }I_{\lambda }\left(
su^{+}\right) =I_{\lambda }\left( s_{0}^{+}u^{+}\right) \geq \alpha
_{\lambda }^{-}.
\end{equation*}%
Since $u\in \mathbf{N}_{\lambda }^{\left( 1\right) },$ we have
\begin{equation*}
\left\Vert u^{+}\right\Vert _{H^{1}}^{2}+\lambda \left( \int_{\mathbb{R}%
^{3}}K(x)\phi _{K,u^{-}}(u^{+})^{2}dx+\int_{\mathbb{R}^{3}}K(x)\phi
_{K,u^{+}}(u^{+})^{2}dx\right) -\int_{\mathbb{R}^{3}}f(x)|u^{+}|^{p}dx=0,
\end{equation*}%
which implies that%
\begin{equation*}
\widehat{h}_{u^{+}}^{\prime }\left( 1\right) =-\lambda \int_{\mathbb{R}%
^{3}}K(x)\phi _{K,u^{-}}(u^{+})^{2}dx\leq 0.
\end{equation*}%
This indicates that $0<s_{0}^{+}\leq 1.$ Finally, we obtain
\begin{eqnarray*}
J_{\lambda }^{+}\left( u^{+},u^{-}\right) &=&\sup_{s\in \left[ 0,s_{\lambda
}^{+}\right] }J_{\lambda }^{+}\left( su^{+},u^{-}\right) \geq \sup_{s\in %
\left[ 0,s_{\lambda }^{+}\right] }I_{\lambda }\left( su^{+}\right) \\
&=&I_{\lambda }\left( s_{0}^{+}u^{+}\right) \geq \alpha _{\lambda }^{-}.
\end{eqnarray*}%
This completes the proof.
\end{proof}

\section{Estimates of energy}

Consider the following autonomous Schr\"{o}dinger-Poisson systems:%
\begin{equation}
\left\{
\begin{array}{ll}
-\Delta u+u+\lambda K_{\infty }\phi u=f_{\infty }\left\vert u\right\vert
^{p-2}u & \text{ in }\mathbb{R}^{3}, \\
-\Delta \phi =K_{\infty }u^{2} & \ \text{in }\mathbb{R}^{3},%
\end{array}%
\right.   \tag{$SP_{\lambda }^{\infty }$}
\end{equation}%
where $\lambda >0$ and $2<p<4.$ By \cite[Theorem 1.3]{SWF1}, there exists $%
\Lambda >0$ such that for each $0<\lambda <\Lambda ,$ system $\left(
SP_{\lambda }^{\infty }\right) $ admits a positive solution $(w_{\lambda
}^{\infty },\phi _{K_{\infty },w_{\lambda }^{\infty }})\in H^{1}(\mathbb{R}%
^{3})\times D^{1,2}(\mathbb{R}^{3})$ satisfying%
\begin{equation}
\alpha _{\lambda }^{\infty }:=I_{\lambda }^{\infty }\left( w_{\lambda
}^{\infty }\right) >\alpha _{0}^{\infty }:=\frac{p-2}{2p}\left( \frac{%
S_{p}^{p}}{f_{\infty }}\right) ^{\frac{2}{p-2}}  \label{3-7}
\end{equation}%
and%
\begin{equation}
\alpha _{\lambda }^{\infty }\rightarrow \alpha _{0}^{\infty }=\frac{p-2}{2p}%
\left( \frac{S_{p}^{p}}{f_{\infty }}\right) ^{\frac{2}{p-2}}\text{ as }%
\lambda \rightarrow 0^{+},  \label{3-9}
\end{equation}%
where $I_{\lambda }^{\infty }$ is the energy functional of system $\left(
SP_{\lambda }^{\infty }\right) .$ Note that conditions $\left( F1\right)
-\left( F2\right) $ and $\left( K1\right) -\left( K2\right) $ satisfy
conditions $\left( D1\right) -\left( D3\right) $ in \cite[Theorem 1.4]{SWF1}%
, and thus for each $0<\lambda <\Lambda ,$ system $\left( SP_{\lambda
}\right) $ admits a positive solution $(v_{\lambda },\phi _{K,v_{\lambda
}})\in H^{1}(\mathbb{R}^{3})\times D^{1,2}(\mathbb{R}^{3})$ satisfying%
\begin{equation}
\frac{p-2}{4p}\left( \frac{S_{p}^{p}}{f_{\max }}\right) ^{\frac{2}{p-2}%
}<\alpha _{\lambda }^{-}:=I_{\lambda }\left( v_{\lambda }\right) <\alpha
_{\lambda }^{\infty }.  \label{3-10}
\end{equation}%
Moreover, by using the Moser's iteration and the De Giorgi's iteration (or
see \cite[Proposition 1]{MT}), we can easily prove that both $w_{\lambda
}^{\infty }$ and $v_{\lambda }$ have exponential decay, and so, both $\phi
_{K_{\infty },w_{\lambda }^{\infty }}$ and $\phi _{K,v_{\lambda }}$ have the
same behavior. That is, for each $0<\varepsilon <1$ there exists $%
C_{\varepsilon }>0$ such that%
\begin{equation*}
v_{\lambda }\left( x\right) ,w_{\lambda }^{\infty }\left( x\right) ,\phi
_{K_{\infty },w_{\lambda }^{\infty }}\left( x\right) ,\phi _{K,v_{\lambda
}}\leq \left( x\right) C_{\varepsilon }\exp \left( -2^{\frac{\varepsilon }{%
1-\varepsilon }}\left( 1+\left\vert x\right\vert \right) ^{1-\varepsilon
}\right) .
\end{equation*}%
Note that%
\begin{equation*}
\left( 1+\left\vert x\right\vert \right) ^{1-\varepsilon }\geq 2^{\frac{%
-\varepsilon }{1-\varepsilon }}\left( 1+\left\vert x\right\vert
^{1-\varepsilon }\right) .
\end{equation*}%
Then, we have%
\begin{equation}
v_{\lambda }\left( x\right) ,w_{\lambda }^{\infty }\left( x\right) ,\phi
_{K_{\infty },w_{\lambda }^{\infty }}\left( x\right) ,\phi _{K,v_{\lambda
}}\left( x\right) \leq C_{\varepsilon }\exp \left( -\left\vert x\right\vert
^{1-\varepsilon }\right) .  \label{3-1}
\end{equation}%
For $n\in \mathbb{N},$ we define the sequence $\{w_{n}\}$ by%
\begin{equation}
w_{n}\left( x\right) =w_{\lambda }^{\infty }(x-ne_{1}),  \label{3-8}
\end{equation}%
where $e_{1}=\left( 1,0,0\right) $. Clearly, $I_{\lambda }^{\infty }\left(
w_{n}\right) =I_{\lambda }^{\infty }\left( w_{\lambda }^{\infty }\right) $
for all $n\in \mathbb{N},$ and by $(\ref{3-1})$ one has%
\begin{equation}
w_{n}\left( x\right) =w_{\lambda }^{\infty }\left( x-ne_{1}\right) \leq
C_{\varepsilon }\exp \left( \left\vert x\right\vert ^{1-\varepsilon
}-n^{1-\varepsilon }\right) .  \label{3-11}
\end{equation}%
Then following \cite{ASS}, we have the following result.

\begin{lemma}
\label{h-1}Suppose that conditions $\left( F1\right) -\left( F2\right) $ and
$\left( K1\right) -\left( K2\right) $ hold. Then for each $0<\varepsilon <1$
there exists $C_{\varepsilon }>0$ such that\newline
$(i) $ $\int_{\mathbb{R}^{3}}K(x)\phi _{K,w_{n}}v_{\lambda }^{2}dx=\int_{%
\mathbb{R}^{3}}K(x)\phi _{K,v_{\lambda }}w_{n}^{2}dx\leq K_{\max
}^{2}C_{\varepsilon }^{2}\exp \left( -2n^{1-\varepsilon }\right) ;$\newline
$(ii) $ $\left\vert \int_{\mathbb{R}^{3}}K(x)\phi _{K,v_{\lambda
}-w_{n}}\left( v_{\lambda }-w_{n}\right) ^{2}dx-\int_{\mathbb{R}%
^{3}}K(x)\phi _{K,v_{\lambda }}v_{\lambda }^{2}dx-\int_{\mathbb{R}%
^{3}}K(x)\phi _{K,w_{n}}w_{n}^{2}dx\right\vert \leq C_{\varepsilon }^{2}\exp
\left( -2n^{1-\varepsilon }\right) .$
\end{lemma}

\begin{proof}
$\left( i\right) $ By Fubini's Theorem, we have%
\begin{equation*}
\int_{\mathbb{R}^{3}}K(x)\phi _{K,w_{n}}v_{\lambda }^{2}dx=\int_{\mathbb{R}%
^{3}}K(x)\phi _{K,v_{\lambda }}w_{n}^{2}dx.
\end{equation*}%
Then it follows from $(\ref{3-1})$ and $(\ref{3-11})$ that%
\begin{eqnarray*}
\int_{\mathbb{R}^{3}}K(x)\phi _{K,w_{n}}v_{\lambda }^{2}dx &=&\int_{\mathbb{R%
}^{3}}K(x)\phi _{K,v_{\lambda }}w_{n}^{2}dx \\
&\leq &K_{\max }^{2}C_{\varepsilon }^{2}\int_{\mathbb{R}^{3}}\exp \left(
-2\left\vert x\right\vert ^{1-\varepsilon }\right) \exp \left( -2\left\vert
x-e_{1}\right\vert ^{1-\varepsilon }\right) dx \\
&\leq &K_{\max }^{2}C_{\varepsilon }^{2}\exp \left( -2n^{1-\varepsilon
}\right) .
\end{eqnarray*}%
$\left( ii\right) $ By part $\left( i\right) ,$ we easily arrive at the
conclusion.
\end{proof}

\begin{lemma}
\label{h-2}Suppose that $2<p<4,$ and conditions $\left( F1\right) -\left(
F2\right) $ and $\left( K1\right) -\left( K2\right) $ hold. Then there
exists a positive number $\lambda _{3}\leq \min \{\widetilde{\lambda }%
,\Lambda \}$ such that for each $0<\lambda <\lambda _{3},$ there exist two
numbers $s_{\lambda }^{\left( 1\right) }$ and $s_{\lambda }^{\left( 2\right)
}$ satisfying
\begin{equation*}
s_{\lambda }<1=s_{\lambda }^{\left( 1\right) }<\left( \frac{2}{4-p}\right) ^{%
\frac{1}{p-2}}s_{\lambda }<s_{\lambda }^{\left( 2\right) }
\end{equation*}%
and $s_{\lambda }^{\left( j\right) }v_{\lambda }\in \mathbf{M}_{\lambda
}^{(j)}$ for $j=1,2,$ where
\begin{equation}
s_{\lambda }=\left( \frac{\left\Vert v_{\lambda }\right\Vert _{H^{1}}^{2}}{%
\int_{\mathbb{R}^{3}}f(x)\left\vert v_{\lambda }\right\vert ^{p}dx}\right) ^{%
\frac{1}{p-2}}.  \label{2-9}
\end{equation}%
Furhtermore, we have
\begin{equation*}
I_{\lambda }\left( v_{\lambda }\right) =\sup_{0\leq t\leq \widehat{s}%
_{\lambda }}I_{\lambda }(sv_{\lambda })
\end{equation*}%
and
\begin{equation*}
I_{\lambda }\left( s_{\lambda }^{\left( 2\right) }v_{\lambda }\right)
=\inf_{t\geq 0}I_{\lambda }\left( sv_{\lambda }\right) <I_{\lambda }\left(
\widehat{s}_{\lambda }v_{\lambda }\right) <0,
\end{equation*}%
where $\left( \frac{p}{2}\right) ^{\frac{1}{p-2}}<\widehat{s}_{\lambda
}=\left( \frac{p}{4-p}\right) ^{\frac{1}{p-2}}s_{\lambda }<\left( \frac{p}{%
4-p}\right) ^{\frac{1}{p-2}}.$
\end{lemma}

\begin{proof}
Similar to the argument of Lemma \ref{h3-5}, one can easily prove that there
exist two numbers $s_{\lambda }^{\left( 1\right) }$ and $s_{\lambda
}^{\left( 2\right) }$ satisfying%
\begin{equation*}
s_{\lambda }<1=s_{\lambda }^{\left( 1\right) }<\left( \frac{2}{4-p}\right) ^{%
\frac{1}{p-2}}s_{\lambda }<s_{\lambda }^{\left( 2\right) }
\end{equation*}%
such that $s_{\lambda }^{\left( j\right) }v_{\lambda }\in \mathbf{M}%
_{\lambda }^{(j)}$ for $j=1,2,$ and $I_{\lambda }\left( s_{\lambda }^{\left(
2\right) }v_{\lambda }\right) =\inf_{t\geq 0}I_{\lambda }(sv_{\lambda }),$
where $s_{\lambda }$ is defined as $(\ref{2-9}).$

Note that%
\begin{eqnarray*}
I_{\lambda }\left( sv_{\lambda }\right)  &=&\frac{s^{2}}{2}\left\Vert
v_{\lambda }\right\Vert _{H^{1}}^{2}+\frac{\lambda s^{4}}{4}\int_{\mathbb{R}%
^{3}}K(x)\phi _{K,v_{\lambda }}v_{\lambda }^{2}dx-\frac{s^{p}}{p}\int_{%
\mathbb{R}^{3}}f(x)\left\vert v_{\lambda }\right\vert ^{p}dx \\
&=&s^{4}\left[ \overline{g}_{v_{\lambda }}\left( s\right) +\frac{\lambda }{4}%
\int_{\mathbb{R}^{3}}K(x)\phi _{K,v_{\lambda }}v_{\lambda }^{2}dx\right] ,
\end{eqnarray*}%
where
\begin{equation*}
\overline{g}_{v_{\lambda }}\left( s\right) =\frac{s^{-2}}{2}\left\Vert
v_{\lambda }\right\Vert _{H^{1}}^{2}-\frac{s^{p-4}}{p}\int_{\mathbb{R}%
^{3}}f(x)\left\vert v_{\lambda }\right\vert ^{p}dx.
\end{equation*}%
Clearly, $I_{\lambda }\left( sv_{\lambda }\right) =0$ if and only if
\begin{equation*}
\overline{g}_{v_{\lambda }}\left( s\right) +\frac{\lambda }{4}\int_{\mathbb{R%
}^{3}}K(x)\phi _{K,v_{\lambda }}v_{\lambda }^{2}dx=0.
\end{equation*}%
By analyzing the functions $\overline{g}_{v_{\lambda }},$ one has
\begin{equation*}
\overline{g}_{v_{\lambda }}\left( \overline{s}_{\lambda }\right) =0,\
\lim_{s\rightarrow 0^{+}}g_{\lambda }(s)=\infty \text{ and}\
\lim_{s\rightarrow \infty }g_{\lambda }(s)=0,
\end{equation*}%
where%
\begin{equation}
\overline{s}_{\lambda }=\left( \frac{p\left\Vert v_{\lambda }\right\Vert
_{H^{1}}^{2}}{2\int_{\mathbb{R}^{3}}f(x)\left\vert v_{\lambda }\right\vert
^{p}dx}\right) ^{\frac{1}{p-2}}.  \label{2-8}
\end{equation}%
Moreover, it is easy to see that
\begin{eqnarray*}
\overline{g}_{v_{\lambda }}^{\prime }\left( s\right)  &=&-s^{-3}\left\Vert
v_{\lambda }\right\Vert _{H^{1}}^{2}+\frac{4-p}{p}s^{p-5}\int_{\mathbb{R}%
^{3}}f(x)\left\vert v_{\lambda }\right\vert ^{p}dx \\
&=&s^{-3}\left( -\left\Vert v_{\lambda }\right\Vert _{H^{1}}^{2}+\frac{4-p}{p%
}s^{p-2}\int_{\mathbb{R}^{3}}f(x)\left\vert v_{\lambda }\right\vert
^{p}dx\right) .
\end{eqnarray*}%
This indicates that $\overline{g}_{v_{\lambda }}\left( s\right) $ is
decreasing when $0<s<\widehat{s}_{\lambda }$ and is increasing when $s>%
\widehat{s}_{\lambda },$ where%
\begin{equation*}
\left( \frac{p}{2}\right) ^{\frac{1}{p-2}}<\widehat{s}_{\lambda }:=\left(
\frac{p}{4-p}\right) ^{\frac{1}{p-2}}s_{\lambda }<\left( \frac{p}{4-p}%
\right) ^{\frac{1}{p-2}}.
\end{equation*}%
Moreover, there holds
\begin{eqnarray*}
\inf_{s>0}\overline{g}_{v_{\lambda }}\left( s\right)  &=&\overline{g}%
_{v_{\lambda }}\left( \widehat{s}_{\lambda }\right) =-\frac{p-2}{2\left(
4-p\right) }\left( \frac{4-p}{p}\right) ^{\frac{2}{p-2}}s_{\lambda
}^{-2}\Vert v_{\lambda }\Vert _{H^{1}}^{2} \\
&<&-\frac{p-2}{2\left( 4-p\right) }\left( \frac{4-p}{p}\right) ^{\frac{2}{p-2%
}}\Vert v_{\lambda }\Vert _{H^{1}}^{2}.
\end{eqnarray*}%
Note that%
\begin{equation*}
\left( \frac{S_{p}^{p}}{f_{\max }}\right) ^{\frac{1}{p-2}}\leq \left\Vert
v_{\lambda }\right\Vert _{H^{1}}<\left( \frac{2S_{p}^{p}}{f_{\max }\left(
4-p\right) }\right) ^{\frac{1}{p-2}}.
\end{equation*}%
Then there exists a positive constant $\lambda _{3}\leq \min \{\widetilde{%
\lambda },\Lambda \}$ such that for each $0<\lambda <\lambda _{3},$
\begin{equation}
\inf_{s>0}\overline{g}_{v_{\lambda }}\left( s\right) =\overline{g}%
_{v_{\lambda }}\left( \widehat{s}_{\lambda }\right) <-\frac{\lambda }{4}%
\int_{\mathbb{R}^{3}}K(x)\phi _{K,v_{\lambda }}v_{\lambda }^{2}dx.
\label{2-13}
\end{equation}%
Thus, there exist two numbers $\overline{s}_{\lambda }^{(j)}(j=1,2)$
satisfying%
\begin{equation*}
\overline{s}_{\lambda }<\overline{s}_{\lambda }^{(1)}<\left( \frac{p}{4-p}%
\right) ^{\frac{1}{p-2}}s_{\lambda }<\overline{s}_{\lambda }^{(2)}
\end{equation*}%
such that
\begin{equation*}
\overline{g}_{v_{\lambda }}\left( \overline{s}_{\lambda }^{(j)}\right)
+\lambda \int_{\mathbb{R}^{3}}K(x)\phi _{v_{\lambda }}v_{\lambda }^{2}dx=0,
\end{equation*}%
namely, $I_{\lambda }\left( \overline{s}_{\lambda }^{(j)}v_{\lambda }\right)
=0,$ where $s_{\lambda }$ and $\overline{s}_{\lambda }$ are defined as $(\ref%
{2-9})$ and $(\ref{2-8})$, respectively. It follows from $(\ref{2-13})$ that%
\begin{eqnarray*}
I_{\lambda }\left( \widehat{s}_{\lambda }v_{\lambda }\right)  &=&I_{\lambda
}\left( \left( \frac{p}{4-p}\right) ^{\frac{1}{p-2}}s_{\lambda }v_{\lambda
}\right)  \\
&=&\left( \left( \frac{p}{4-p}\right) ^{\frac{1}{p-2}}s_{\lambda }\right)
^{4}\left[ g_{\lambda }\left( \left( \frac{p}{4-p}\right) ^{\frac{1}{p-2}%
}s_{\lambda }\right) +\frac{\lambda }{4}\int_{\mathbb{R}^{3}}K(x)\phi
_{K,v_{\lambda }}v_{\lambda }^{2}dx\right]  \\
&<&0,
\end{eqnarray*}%
which leads to
\begin{equation*}
\inf_{t\geq 0}I_{\lambda }\left( tv_{\lambda }\right) <I_{\lambda }\left(
\widehat{s}_{\lambda }v_{\lambda }\right) <0,
\end{equation*}%
and%
\begin{equation*}
I_{\lambda }\left( v_{\lambda }\right) =I_{\lambda }\left( s_{\lambda
}^{\left( 1\right) }v_{\lambda }\right) =\sup_{0\leq t\leq \widehat{s}%
_{\lambda }}I_{\lambda }(sv_{\lambda }).
\end{equation*}%
This completes the proof.
\end{proof}

\begin{lemma}
\label{h-3}Suppose that $2<p<4$ and conditions $\left( F1\right) ,\left(
F2\right) ,$ $\left( K1\right) $ and $\left( K2\right) $ hold. Then there
exist two positive number $\lambda _{4}\leq \min \{\widetilde{\lambda }%
,\Lambda \}$ and $n_{0}\in \mathbb{N}$ such that for every $0<\lambda
<\lambda _{4}$ and $n\geq n_{0},$ there exist two numbers $t_{n}^{\left(
1\right) }$ and $t_{n}^{\left( 2\right) }$ satisfying
\begin{equation*}
t_{n}^{\infty }<1=t_{n}^{\left( 1\right) }<\left( \frac{2}{4-p}\right) ^{%
\frac{1}{p-2}}t_{n}^{\infty }<t_{n}^{\left( 2\right) }
\end{equation*}%
and $t_{n}^{\left( j\right) }w_{n}\in \mathbf{M}_{\lambda }^{(j)}$ for $%
j=1,2,$ where
\begin{equation}
t_{n}^{\infty }=\left( \frac{\left\Vert w_{n}\right\Vert _{H^{1}}^{2}}{\int_{%
\mathbb{R}^{3}}f_{\infty }\left\vert w_{n}\right\vert ^{p}dx}\right) ^{\frac{%
1}{p-2}}.  \label{5-1}
\end{equation}%
Furthermore, we have
\begin{equation*}
I_{\lambda }\left( w_{n}\right) =\sup_{0\leq t\leq \widehat{t}%
_{n}}I_{\lambda }(tw_{n}),
\end{equation*}%
and
\begin{equation*}
I_{\lambda }\left( t_{n}^{\left( 2\right) }w_{n}\right) =\inf_{t\geq
0}I_{\lambda }\left( tw_{n}\right) <I_{\lambda }\left( \widehat{t}%
_{n}w_{n}\right) <0,
\end{equation*}%
where $\left( \frac{p}{2}\right) ^{\frac{1}{p-2}}<\widehat{t}_{n}<\left(
\frac{p}{4-p}\right) ^{\frac{1}{p-2}}.$
\end{lemma}

\begin{proof}
Similar to the argument of Lemma \ref{h3-5}, it is easy to prove that there
exist two numbers $t_{n}^{\left( 1\right) }$ and $t_{n}^{\left( 2\right) }$
satisfying
\begin{equation*}
t_{n}^{\infty }<1=t_{n}^{\left( 1\right) }<\left( \frac{2}{4-p}\right) ^{%
\frac{1}{p-2}}t_{n}^{\infty }<t_{n}^{\left( 2\right) }
\end{equation*}%
and $t_{n}^{\left( j\right) }w_{n}\in \mathbf{M}_{\lambda }^{(j)}$ for $%
j=1,2,$ and $I_{\lambda }\left( t_{n}^{\left( 2\right) }w_{n}\right)
=\inf_{t\geq 0}I_{\lambda }\left( tw_{n}\right) ,$ where $t_{n}^{\infty }$
is defined as $(\ref{5-1})$ satisfying
\begin{equation}
\left( \frac{4-p}{2}\right) ^{\frac{1}{p-2}}<t_{n}^{\infty }<1.  \label{5-13}
\end{equation}

Note that%
\begin{equation*}
\int_{\mathbb{R}^{3}}\left( f(x)-f_{\infty }\right) \left\vert
w_{n}\right\vert ^{p}dx=o\left( 1\right)
\end{equation*}%
and%
\begin{equation*}
\int_{\mathbb{R}^{3}}K(x)\phi _{K,w_{n}}w_{n}^{2}dx-\int_{\mathbb{R}%
^{3}}K_{\infty }\phi _{K_{\infty },w_{n}}w_{n}^{2}dx=o(1).
\end{equation*}%
Then it follows from $\left( \ref{5-1}\right) -\left( \ref{5-13}\right) $
that there exists $n_{0}\in \mathbb{N}$ such that for any $n\geq n_{0},$%
\begin{equation*}
\left( \frac{4-p}{2}\right) ^{\frac{1}{p-2}}<t_{n}:=\left( \frac{\left\Vert
w_{n}\right\Vert _{H^{1}}^{2}}{\int_{\mathbb{R}^{3}}f(x)\left\vert
w_{n}\right\vert ^{p}dx}\right) ^{\frac{1}{p-2}}<1.
\end{equation*}%
Moreover, by $(\ref{3-7}),$ one has $I_{\lambda }\left( w_{n}\right) =\alpha
_{\lambda }^{\infty }$ for any $n\geq n_{0}.$ It is easy to see that%
\begin{eqnarray}
I_{\lambda }\left( tw_{n}\right)  &=&\frac{t^{2}}{2}\left\Vert
w_{n}\right\Vert _{H^{1}}^{2}+\frac{\lambda t^{4}}{4}\int_{\mathbb{R}%
^{3}}K(x)\phi _{K,w_{n}}w_{n}^{2}dx-\frac{t^{p}}{p}\int_{\mathbb{R}%
^{3}}f(x)\left\vert w_{n}\right\vert ^{p}dx  \notag \\
&=&t^{4}\left[ g_{w_{n}}\left( t\right) +\frac{\lambda }{4}\int_{\mathbb{R}%
^{3}}K(x)\phi _{K,w_{n}}w_{n}^{2}dx\right] ,  \label{5-14}
\end{eqnarray}%
where
\begin{equation*}
g_{w_{n}}\left( t\right) =\frac{t^{-2}}{2}\left\Vert w_{n}\right\Vert
_{H^{1}}^{2}-\frac{t^{p-4}}{p}\int_{\mathbb{R}^{3}}f(x)\left\vert
w_{n}\right\vert ^{p}dx.
\end{equation*}%
Clearly, $I_{\lambda }\left( tw_{n}\right) =0$ if and only if
\begin{equation*}
g_{w_{n}}\left( t\right) +\frac{\lambda }{4}\int_{\mathbb{R}^{3}}K(x)\phi
_{K,w_{n}}w_{n}^{2}dx=0.
\end{equation*}%
By analyzing the functions $g_{w_{n}}$ one has
\begin{equation*}
g_{w_{n}}\left( \widetilde{t}_{n}\right) =0,\ \lim_{t\rightarrow
0^{+}}g_{w_{n}}(t)=\infty \text{ and}\ \lim_{t\rightarrow \infty
}g_{w_{n}}(t)=0,
\end{equation*}%
where%
\begin{equation}
\widetilde{t}_{n}=\left( \frac{p\left\Vert w_{n}\right\Vert _{H^{1}}^{2}}{%
2\int_{\mathbb{R}^{3}}f(x)\left\vert w_{n}\right\vert ^{p}dx}\right) ^{\frac{%
1}{p-2}}.  \label{5-4}
\end{equation}%
A direct calculation shows that
\begin{equation*}
g_{w_{n}}^{\prime }\left( t\right) =t^{-3}\left( -\left\Vert
w_{n}\right\Vert _{H^{1}}^{2}+\frac{4-p}{p}t^{p-2}\int_{\mathbb{R}%
^{3}}f(x)\left\vert w_{n}\right\vert ^{p}dx\right) .
\end{equation*}%
This implies that $g_{w_{n}}\left( t\right) $ is decreasing when $0<t<%
\widehat{t}_{n}$ and is increasing when $t>\widehat{t}_{n},$ where
\begin{equation}
\left( \frac{p}{2}\right) ^{\frac{1}{p-2}}<\widehat{t}_{n}:=\left( \frac{p}{%
4-p}\right) ^{\frac{1}{p-2}}t_{n}<\left( \frac{p}{4-p}\right) ^{\frac{1}{p-2}%
}.  \label{5-15}
\end{equation}%
Moreover, there holds
\begin{eqnarray*}
\inf_{t>0}g_{w_{n}}\left( t\right)  &=&g_{w_{n}}\left( \widehat{t}%
_{n}\right) =-\frac{p-2}{2\left( 4-p\right) }\left( \frac{4-p}{p}\right) ^{%
\frac{2}{p-2}}t_{n}^{-2}\Vert w_{n}\Vert _{H^{1}}^{2} \\
&<&-\frac{p-2}{2\left( 4-p\right) }\left( \frac{4-p}{p}\right) ^{\frac{2}{p-2%
}}\Vert w_{n}\Vert _{H^{1}}^{2}.
\end{eqnarray*}%
Since%
\begin{equation*}
\left( \frac{S_{p}^{p}}{f_{\max }}\right) ^{\frac{1}{p-2}}\leq \left\Vert
w_{n}\right\Vert _{H^{1}}<\left( \frac{2S_{p}^{p}}{f_{\max }\left(
4-p\right) }\right) ^{\frac{1}{p-2}},
\end{equation*}%
there exists a positive number $\lambda _{4}\leq \min \{\widetilde{\lambda }%
,\Lambda \}$ such that%
\begin{equation*}
\inf_{t>0}g_{w_{n}}\left( t\right) =g_{w_{n}}\left( \widehat{t}_{n}\right) <-%
\frac{\lambda }{4}\int_{\mathbb{R}^{3}}K(x)\phi _{K,w_{n}}w_{n}^{2}dx
\end{equation*}%
for all $0<\lambda <\lambda _{4}.$ Thus, there are two numbers $\widehat{t}%
_{n}^{(1)}$ and $\widehat{t}_{n}^{\left( 2\right) }$ satisfying $\widetilde{t%
}_{n}<\widehat{t}_{n}^{\left( 1\right) }<\widehat{t}_{n}<\widehat{t}%
_{n}^{\left( 2\right) }$ such that
\begin{equation*}
g_{w_{n}}\left( \widehat{t}_{n}^{\left( j\right) }\right) +\frac{\lambda }{4}%
\int_{\mathbb{R}^{3}}K(x)\phi _{w_{n}}w_{n}^{2}dx=0\text{ for }j=1,2,
\end{equation*}%
i.e., $I_{\lambda }\left( \widehat{t}_{n}^{\left( j\right) }w_{n}\right) =0,$
where $\widetilde{t}_{n}$ and $\widehat{t}_{n}$ are defined as $(\ref{5-4})$
and $(\ref{5-15})$, respectively. It follows from $(\ref{5-14})$ that%
\begin{eqnarray*}
I_{\lambda }\left( \widehat{t}_{n}w_{n}\right)  &=&\left( \left( \frac{p}{4-p%
}\right) ^{\frac{1}{p-2}}t_{n}\right) ^{4}\left[ g_{w_{n}}\left( \left(
\frac{p}{4-p}\right) ^{\frac{1}{p-2}}t_{n}\right) +\frac{\lambda }{4}\int_{%
\mathbb{R}^{3}}K(x)\phi _{K,w_{n}}w_{n}^{2}dx\right]  \\
&<&0,
\end{eqnarray*}%
which gives $\inf_{t\geq 0}I_{\lambda }\left( tw_{n}\right) <I_{\lambda
}\left( \widehat{t}_{n}w_{n}\right) <0$ and $I_{\lambda }\left( w_{n}\right)
=\sup_{0\leq t\leq \widehat{t}_{n}}I_{\lambda }(tw_{n}).$ This completes the
proof.
\end{proof}

\begin{lemma}
\label{h-4}Suppose that $2<p<4,$ and conditions $\left( F1\right) -\left(
F2\right) $ and $\left( K1\right) -\left( K2\right) $ hold. Then for any $%
0<\lambda <\min \{\lambda _{3},\lambda _{4}\}$ and $n\geq n_{0},$ there holds%
\begin{equation*}
I_{\lambda }(v_{\lambda }-w_{n})>\sup_{\left( s,t\right) \in \partial
\left\{ \left[ 0,\widehat{s}_{\lambda }\right] \times \left[ 0,\widehat{t}%
_{n}\right] \right\} }I_{\lambda }(sv_{\lambda }-tw_{n}),
\end{equation*}%
where $\widehat{s}_{\lambda }$ and $\widehat{t}_{n}$ are defined in Lemmas %
\ref{h-2} and \ref{h-3}, respectively. Furthermore, we have
\begin{equation*}
\lim_{n\rightarrow \infty }I_{\lambda }(v_{\lambda }-w_{n})=\alpha _{\lambda
}^{-}+\alpha _{\lambda }^{\infty }.
\end{equation*}
\end{lemma}

\begin{proof}
Note that $1<\left( \frac{p}{2}\right) ^{\frac{1}{p-2}}<\widehat{s}_{\lambda
},\widehat{t}_{n}<\left( \frac{p}{4-p}\right) ^{\frac{1}{p-2}}$ for all $%
0<\lambda <\min \{\lambda _{3},\lambda _{4}\}$ and $n\geq n_{0}$ by Lemmas %
\ref{h-2} and \ref{h-3}. Then for all $\left( s,t\right) \in \left[ 0,%
\widehat{s}_{\lambda }\right] \times \left[ 0,\widehat{t}_{n}\right] ,$ by
virtue of Lemma \ref{h-1} and $(\ref{3-8}),$ we have%
\begin{eqnarray}
&&\int_{\mathbb{R}^{3}}K(x)\phi _{K,\left( sv_{\lambda }-tw_{n}\right)
}\left( sv_{\lambda }-tw_{n}\right) ^{2}dx  \notag \\
&=&s^{4}\int_{\mathbb{R}^{3}}K(x)\phi _{K,v_{\lambda }}v_{\lambda
}^{2}dx+t^{4}\int_{\mathbb{R}^{3}}K_{\infty }\phi _{K_{\infty },w_{\lambda
}^{\infty }}\left( w_{\lambda }^{\infty }\right) ^{2}dx+o(1).  \label{5-8}
\end{eqnarray}%
Moreover, using the fact of $w_{n}\rightarrow 0$ a.e. in $\mathbb{R}^{3}$
and \cite[Brezis-Lieb Lemma]{BLi} gives%
\begin{equation}
\left\Vert sv_{\lambda }-tw_{n}\right\Vert _{H^{1}}^{2}=s^{2}\left\Vert
v_{\lambda }\right\Vert _{H^{1}}^{2}+t^{2}\left\Vert w_{n}\right\Vert
_{H^{1}}^{2}+o(1)  \label{5-2}
\end{equation}%
and%
\begin{equation}
\int_{\mathbb{R}^{3}}f(x)\left\vert sv_{\lambda }-tw_{n}\right\vert
^{p}dx=s^{p}\int_{\mathbb{R}^{3}}f(x)\left\vert v_{\lambda }\right\vert
^{p}dx+t^{p}\int_{\mathbb{R}^{3}}f_{\infty }\left\vert w_{\lambda }^{\infty
}\right\vert ^{p}dx+o(1).  \label{5-9}
\end{equation}%
It follows from $\left( \ref{5-8}\right) -\left( \ref{5-9}\right) $ that%
\begin{eqnarray*}
I_{\lambda }\left( v_{\lambda }-w_{n}\right)  &=&\frac{1}{2}\left\Vert
v_{\lambda }-w_{n}\right\Vert _{H^{1}}^{2}+\frac{\lambda }{4}\int_{\mathbb{R}%
^{3}}K(x)\phi _{K,\left( v_{\lambda }-w_{n}\right) }\left( v_{\lambda
}-w_{n}\right) ^{2}dx \\
&&-\frac{1}{p}\int_{\mathbb{R}^{3}}f(x)\left\vert v_{\lambda
}-w_{n}\right\vert ^{p}dx \\
&=&\frac{1}{2}\left\Vert v_{\lambda }\right\Vert _{H^{1}}^{2}+\frac{\lambda
}{4}\int_{\mathbb{R}^{3}}K(x)\phi _{K,v_{\lambda }}v_{\lambda }^{2}dx-\frac{1%
}{p}\int_{\mathbb{R}^{3}}f(x)\left\vert v_{\lambda }\right\vert ^{p}dx \\
&&+\frac{1}{2}\left\Vert w_{\lambda }^{\infty }\right\Vert _{H^{1}}^{2}+%
\frac{\lambda }{4}\int_{\mathbb{R}^{3}}K_{\infty }\phi _{K_{\infty
},w_{\lambda }^{\infty }}\left( w_{\lambda }^{\infty }\right) ^{2}dx-\frac{1%
}{p}\int_{\mathbb{R}^{3}}f_{\infty }\left\vert w_{\lambda }^{\infty
}\right\vert ^{p}dx+o\left( 1\right)  \\
&=&\alpha _{\lambda }^{-}+\alpha _{\lambda }^{\infty }+o(1).
\end{eqnarray*}%
Thus, by Lemmas \ref{h-2} and \ref{h-3}, for all $0<\lambda <\min \{\lambda
_{3},\lambda _{4}\}$ and $n\geq n_{0},$ there holds%
\begin{equation*}
I_{\lambda }(v_{\lambda }-w_{n})>\sup_{\left( s,t\right) \in D}I_{\lambda
}(sv_{\lambda }-tw_{n}),
\end{equation*}%
where $D=\left( \left[ 0,\widehat{s}_{\lambda }\right] \times \left\{
0\right\} \right) \cup \left( \left\{ 0\right\} \times \left[ 0,\widehat{t}%
_{n}\right] \right) .$ Similarly, we also get%
\begin{equation*}
I_{\lambda }(v_{\lambda }-w_{n})>\sup_{t\in \left[ 0,\widehat{t}_{n}\right]
}I_{\lambda }(\widehat{s}_{\lambda }v_{\lambda }-tw_{n})
\end{equation*}%
and%
\begin{equation*}
I_{\lambda }(v_{n}-w_{n})>\sup_{s\in \left[ 0,\widehat{s}_{\lambda }\right]
}I_{\lambda }\left( sv_{n}-\widehat{t}_{n}w_{n}\right) .
\end{equation*}%
These imply that
\begin{equation*}
I_{\lambda }\left( v_{\lambda }-w_{n}\right) >\sup_{\left( s,t\right) \in
\partial \left\{ \left[ 0,\widehat{s}_{\lambda }\right] \times \left[ 0,%
\widehat{t}_{n}\right] \right\} }I_{\lambda }(sv_{n}-tw_{n}).
\end{equation*}%
This completes the proof.
\end{proof}

For all $\left( s,t\right) \in \left[ 0,\widehat{s}_{\lambda }\right] \times %
\left[ 0,\widehat{t}_{n}\right] ,$ we define%
\begin{eqnarray*}
\widehat{h}\left( s,t\right)  &=&I_{\lambda }\left( sv_{\lambda
}-tw_{n}\right)  \\
&=&\frac{1}{2}\left\Vert sv_{\lambda }-tw_{n}\right\Vert _{H^{1}}^{2}+\frac{%
\lambda }{4}\int_{\mathbb{R}^{3}}K(x)\phi _{K,\left( sv_{\lambda
}-tw_{n}\right) }\left( sv_{\lambda }-tw_{n}\right) ^{2}dx \\
&&-\frac{1}{p}\int_{\mathbb{R}^{3}}f(x)\left\vert sv_{\lambda
}-tw_{n}\right\vert ^{p}dx.
\end{eqnarray*}%
A direct calculation shows that%
\begin{eqnarray*}
\frac{\partial }{\partial s}\widehat{h}\left( s,t\right)  &=&\left\langle
I_{\lambda }^{\prime }\left( sv_{\lambda }-tw_{n}\right) ,v_{\lambda
}\right\rangle  \\
&=&s\left\Vert v_{\lambda }\right\Vert _{H^{1}}^{2}+\lambda st^{2}\int_{%
\mathbb{R}^{3}}K(x)\phi _{K,w_{n}}v_{\lambda }^{2}dx+\lambda s^{3}\int_{%
\mathbb{R}^{3}}K(x)\phi _{K,v_{n}}v_{\lambda }^{2}dx \\
&&-s^{p-1}\int_{\mathbb{R}^{3}}f(x)\left\vert v_{\lambda }\right\vert
^{p}dx+o(1)
\end{eqnarray*}%
and%
\begin{eqnarray*}
\frac{\partial }{\partial t}\widehat{h}\left( s,t\right)  &=&\left\langle
I_{\lambda }^{\prime }\left( sv_{\lambda }-tw_{n}\right)
,-w_{n}\right\rangle  \\
&=&t\left\Vert w_{n}\right\Vert _{H^{1}}^{2}+\lambda s^{2}t\int_{\mathbb{R}%
^{3}}K(x)\phi _{K,v_{\lambda }}w_{n}^{2}dx+\lambda t^{3}\int_{\mathbb{R}%
^{3}}K(x)\phi _{K,w_{n}}w_{n}^{2}dx \\
&&-t^{p-1}\int_{\mathbb{R}^{3}}f(x)\left\vert w_{n}\right\vert ^{p}dx+o(1).
\end{eqnarray*}%
Then we have the following result.

\begin{proposition}
\label{h-5}Suppose that $2<p<4,$ and conditions $\left( F1\right) -\left(
F2\right) $ and $\left( K1\right) -\left( K2\right) $ hold. Then there exist
two positive numbers $\lambda ^{\ast }\leq \min \{\lambda _{3},\lambda _{4}\}
$ and $n^{\ast }\in \mathbb{N}$ such that for every $0<\lambda <\lambda
^{\ast }$ and $n\geq n^{\ast },$ there exists $\left( s_{\lambda }^{\ast
},t_{n}^{\ast }\right) \in \left( 0,\widehat{s}_{\lambda }\right) \times
\left( 0,\widehat{t}_{n}\right) $ such that
\begin{equation*}
I_{\lambda }\left( s_{\lambda }^{\ast }v_{\lambda }-t_{n}^{\ast
}w_{n}\right) =\sup_{\left( s,t\right) \in \left[ 0,\widehat{s}_{\lambda }%
\right] \times \left[ 0,\widehat{t}_{n}\right] }I_{\lambda }\left(
sv_{\lambda }-tw_{n}\right) <\alpha _{\lambda }^{-}+\alpha _{\lambda
}^{\infty },
\end{equation*}%
and $s_{\lambda }^{\ast }v_{\lambda }-t_{n}^{\ast }w_{n}\in \mathbf{N}%
_{\lambda }^{(1)}.$
\end{proposition}

\begin{proof}
It follows from Lemmas \ref{h-2} and \ref{h-3} that $1<\left( \frac{p}{2}%
\right) ^{\frac{1}{p-2}}<\widehat{s}_{\lambda },\widehat{t}_{n}<\left( \frac{%
p}{4-p}\right) ^{\frac{1}{p-2}}$ for all $0<\lambda <\min \{\lambda
_{3},\lambda _{4}\}$ and $n\geq n_{0}.$ By Lemma \ref{h-4}, there exists $%
\left( s_{\lambda }^{\ast },t_{n}^{\ast }\right) \in \left( 0,\widehat{s}%
_{\lambda }\right) \times \left( 0,\widehat{t}_{n}\right) $ such that%
\begin{equation*}
I_{\lambda }\left( s_{\lambda }^{\ast }v_{\lambda }-t_{n}^{\ast
}w_{n}\right) =\sup_{\left( s,t\right) \in \left[ 0,\widehat{s}_{\lambda }%
\right] \times \left[ 0,\widehat{t}_{n}\right] }I_{\lambda }(sv_{\lambda
}-tw_{n})\geq I_{\lambda }(v_{\lambda }-w_{n}),
\end{equation*}%
and
\begin{eqnarray}
\frac{\partial }{\partial s}\widehat{h}\left( s_{\lambda }^{\ast
},t_{n}^{\ast }\right) &=&\left\langle I_{\lambda }^{\prime }\left(
s_{\lambda }^{\ast }v_{\lambda }-t_{n}^{\ast }w_{n}\right) ,v_{\lambda
}\right\rangle =0,  \label{5-17} \\
\frac{\partial }{\partial t}\widehat{h}\left( s_{\lambda }^{\ast
},t_{n}^{\ast }\right) &=&\left\langle I_{\lambda }^{\prime }\left(
s_{\lambda }^{\ast }v_{\lambda }-t_{n}^{\ast }w_{n}\right)
,-w_{n}\right\rangle =0.  \label{5-18}
\end{eqnarray}%
This implies that
\begin{equation*}
\left\langle I_{\lambda }^{\prime }\left( s_{\lambda }^{\ast }v_{\lambda
}-t_{n}^{\ast }w_{n}\right) ,s_{\lambda }^{\ast }v_{\lambda }-t_{n}^{\ast
}w_{n}\right\rangle =0,
\end{equation*}%
i.e., $s_{\lambda }^{\ast }v_{\lambda }-t_{n}^{\ast }w_{n}\in \mathbf{M}%
_{\lambda }.$

Next, we show that%
\begin{equation}
I_{\lambda }\left( s_{\lambda }^{\ast }v_{\lambda }-t_{n}^{\ast
}w_{n}\right) =\sup_{\left( s,t\right) \in \left[ 0,\widehat{s}_{\lambda }%
\right] \times \left[ 0,\widehat{t}_{n}\right] }I_{\lambda }\left(
sv_{\lambda }-tw_{n}\right) <\alpha _{\lambda }^{-}+\alpha _{\lambda
}^{\infty }.  \label{5-20}
\end{equation}%
By virtue of Lemma \ref{h-1}, one has%
\begin{eqnarray}
I_{\lambda }\left( s_{\lambda }^{\ast }v_{\lambda }-t_{n}^{\ast
}w_{n}\right) &=&\frac{1}{2}\left\Vert s_{\lambda }^{\ast }v_{\lambda
}-t^{\ast }w_{n}\right\Vert _{H^{1}}^{2}-\frac{1}{p}\int_{\mathbb{R}%
^{3}}f(x)\left\vert s_{\lambda }^{\ast }v_{\lambda }-t^{\ast
}w_{n}\right\vert ^{p}dx  \notag \\
&&+\frac{\lambda }{4}\int_{\mathbb{R}^{3}}K(x)\phi _{K,\left( s_{\lambda
}^{\ast }v_{\lambda }-t^{\ast }w_{n}\right) }\left( s_{\lambda }^{\ast
}v_{\lambda }-t^{\ast }w_{n}\right) ^{2}dx  \notag \\
&\leq &I_{\lambda }\left( s_{\lambda }^{\ast }v_{\lambda }\right)
+I_{\lambda }^{\infty }\left( t_{n}^{\ast }w_{n}\right) +\frac{\lambda }{4}%
C_{\varepsilon }^{2}\exp \left( -2n^{1-\varepsilon }\right)  \notag \\
&&-\frac{\left( t_{n}^{\ast }\right) ^{p}}{p}\int_{\mathbb{R}^{3}}(f\left(
x\right) -f_{\infty })\left\vert w_{n}\right\vert ^{p}dx  \notag \\
&&-\frac{1}{p}\int_{\mathbb{R}^{3}}f(x)\left( \left\vert s_{\lambda }^{\ast
}v_{\lambda }-t_{n}^{\ast }w_{n}\right\vert ^{p}-\left\vert s_{\lambda
}^{\ast }v_{\lambda }\right\vert ^{p}-\left\vert t_{n}^{\ast
}w_{n}\right\vert ^{p}\right) dx.  \label{5-3}
\end{eqnarray}%
Since $I_{\lambda }(v_{\lambda })=\sup_{0\leq t\leq \widehat{s}_{\lambda
}}I_{\lambda }(sv_{\lambda })$ and $I_{\lambda }^{\infty
}(w_{n})=\sup_{0\leq t\leq \widehat{t}_{n}}I_{\lambda }^{\infty }(tw_{n}),$
we have
\begin{equation}
I_{\lambda }\left( s_{\lambda }^{\ast }v_{\lambda }\right) \leq \alpha
_{\lambda }^{-}\text{ and }I_{\lambda }^{\infty }\left( t_{n}^{\ast
}w_{n}\right) \leq \alpha _{\lambda }^{\infty }.  \label{5-16}
\end{equation}%
Using the inequality%
\begin{equation*}
\left\vert c-d\right\vert ^{p}>c^{p}+d^{p}-C_{\ast }\left( p\right) \left(
c^{p-1}d+cd^{p-1}\right)
\end{equation*}%
for all $c,d>0$ and for some constant $C^{\ast }\left( p\right) >0,$
together with $(\ref{5-3})-(\ref{5-16}),$ leads to%
\begin{eqnarray}
I_{\lambda }\left( s_{\lambda }^{\ast }v_{\lambda }-t_{n}^{\ast
}w_{n}\right) &\leq &\alpha _{\lambda }^{-}+\alpha _{\lambda }^{\infty }+%
\frac{\lambda }{4}C_{\varepsilon }^{2}\exp \left( -2n^{1-\varepsilon
}\right) -\frac{\left( t_{n}^{\ast }\right) ^{p}}{p}\int_{\mathbb{R}%
^{3}}(f\left( x\right) -f_{\infty })\left\vert w_{n}\right\vert ^{p}dx
\notag \\
&&+\frac{C^{\ast }\left( p\right) }{p}\int_{\mathbb{R}^{3}}f\left( x\right)
\left( \left\vert s_{\lambda }^{\ast }v_{\lambda }\right\vert
^{p-1}t_{n}^{\ast }w_{n}+s_{\lambda }^{\ast }v_{\lambda }\left\vert
t_{n}^{\ast }w_{n}\right\vert ^{p-1}\right) dx  \notag \\
&\leq &\alpha _{\lambda }^{-}+\alpha _{\lambda }^{\infty }+\frac{\lambda }{4}%
C_{\varepsilon }^{2}\exp \left( -2n^{1-\varepsilon }\right) -\frac{\left(
t_{n}^{\ast }\right) ^{p}}{p}\int_{\mathbb{R}^{3}}(f\left( x\right)
-f_{\infty })\left\vert w_{n}\right\vert ^{p}dx  \notag \\
&&+\frac{C^{\ast }\left( p\right) }{p}\left( \frac{p}{4-p}\right) ^{\frac{p}{%
p-2}}\int_{\mathbb{R}^{3}}f\left( x\right) \left( \left\vert v_{\lambda
}\right\vert ^{p-1}w_{n}+v_{\lambda }\left\vert w_{n}\right\vert
^{p-1}\right) dx.  \label{5-19}
\end{eqnarray}%
By condition $(F2),$ one has%
\begin{eqnarray}
\int_{\mathbb{R}^{3}}(f\left( x\right) -f_{\infty })\left\vert
w_{n}\right\vert ^{p}dx &\geq &d_{0}\int_{\mathbb{R}^{3}}\exp \left(
-\left\vert x+ne_{1}\right\vert ^{r_{f}}\right) (w_{\lambda }^{\infty
})^{p}\left( x\right) dx  \notag \\
&\geq &\left( \min_{x\in B_{1}\left( 0\right) }\left( w_{\lambda }^{\infty
}\right) ^{p}\right) \int_{B_{1}\left( 0\right) }\exp \left( -\left\vert
x+ne_{1}\right\vert ^{r_{f}}\right) dx  \notag \\
&\geq &\left( \min_{x\in B_{1}\left( 0\right) }\left( w_{\lambda }^{\infty
}\right) ^{p}\right) d_{0}\int_{B_{1}\left( 0\right) }\exp \left(
-\left\vert x\right\vert ^{r_{f}}-|e_{1}|n^{r_{f}}\right) dx  \notag \\
&=&\left( \min_{x\in B_{1}\left( 0\right) }\left( w_{\lambda }^{\infty
}\right) ^{p}\right) D_{0}\exp \left( -n^{r_{f}}\right) .  \label{5-10}
\end{eqnarray}%
Moreover, by \cite[Lemma 4.6]{Lin}, there exists $n_{1}>0$ such that for all
$n>n_{1},$%
\begin{eqnarray}
\int_{\mathbb{R}^{3}}f\left( x\right) \left\vert v_{\lambda }\right\vert
^{p-1}w_{n}dx &\leq &f_{\max }\int_{\mathbb{R}^{3}}\exp \left( -\left(
p-1\right) |x|^{1-\varepsilon }\right) \exp \left(
-|x-e_{1}n|^{1-\varepsilon }\right) dx  \notag \\
&\leq &\overline{C}_{\varepsilon ,1}\exp \left( -n^{1-\varepsilon }\right)
\text{ for some }\overline{C}_{\varepsilon ,1}>0.  \label{5-11}
\end{eqnarray}%
Similarly, we also obtain that there exists $n_{2}>0$ such that for all $%
n>n_{2},$%
\begin{equation}
\int_{\mathbb{R}^{3}}f\left( x\right) v_{\lambda }\left\vert
w_{n}\right\vert ^{p-1}dx\leq \overline{C}_{\varepsilon ,2}\exp \left(
-n^{1-\varepsilon }\right) \text{ for some }\overline{C}_{\varepsilon ,2}>0.
\label{5-12}
\end{equation}%
Hence, by $(\ref{5-19})-(\ref{5-12}),$ we may take $0<\varepsilon <1-r_{f}$
and $n^{\ast }\geq \max \left\{ n_{0},n_{1},n_{2}\right\} $ such that for
every $0<\lambda <\min \{\lambda _{3},\lambda _{4}\}$ and $n\geq n^{\ast },$
there holds%
\begin{eqnarray*}
I_{\lambda }\left( s_{\lambda }^{\ast }v_{\lambda }-t_{n}^{\ast
}w_{n}\right) &\leq &\alpha _{\lambda }^{-}+\alpha _{\lambda }^{\infty }+%
\overline{C}_{\varepsilon }\exp \left( -n^{1-\varepsilon }\right) -C_{0}\exp
\left( -n^{r_{f}}\right) \\
&<&\alpha _{\lambda }^{-}+\alpha _{\lambda }^{\infty },
\end{eqnarray*}%
where $\overline{C}_{\varepsilon }$ and $C_{0}$ are two positive constants.

Finally, we claim that $s_{\lambda }^{\ast }v_{\lambda }-t_{n}^{\ast
}w_{n}\in \mathbf{N}_{\lambda }^{(1)}.$ Note that%
\begin{equation*}
C\left( p\right) >\left\{
\begin{array}{ll}
\frac{\sqrt{e}\left( p-2\right) }{p}, & \text{ if }2<p\leq 3, \\
\frac{e\left( p-2\right) }{2p}, & \text{ if }3<p<4.%
\end{array}%
\right.
\end{equation*}%
Then from $(\ref{3-9})-(\ref{3-10})$ and $(\ref{5-20})$ it follows that for $%
\lambda >0$ sufficiently small,
\begin{equation*}
I_{\lambda }\left( s_{\lambda }^{\ast }v_{\lambda }-t_{n}^{\ast
}w_{n}\right) <\alpha _{\lambda }^{-}+\alpha _{\lambda }^{\infty }<2\alpha
_{\lambda }^{\infty }<C\left( p\right) \left( \frac{S_{p}^{p}}{f_{\infty }}%
\right) ^{\frac{2}{p-2}},
\end{equation*}%
and so we can conclude that either $s_{\lambda }^{\ast }v_{\lambda
}-t_{n}^{\ast }w_{n}\in \mathbf{M}_{\lambda }^{(1)}$ or $s_{\lambda }^{\ast
}v_{\lambda }-t_{n}^{\ast }w_{n}\in \mathbf{M}_{\lambda }^{(2)}.$ If $%
s_{\lambda }^{\ast }v_{\lambda }-t_{n}^{\ast }w_{n}\in \mathbf{M}_{\lambda
}^{(2)},$ then by $(\ref{2-1})$ and $(\ref{2-4}),$ we have%
\begin{eqnarray}
I_{\lambda }\left( s_{\lambda }^{\ast }v_{\lambda }-t_{n}^{\ast
}w_{n}\right)  &=&\frac{p-2}{2p}\left\Vert s_{\lambda }^{\ast }v_{\lambda
}-t_{n}^{\ast }w_{n}\right\Vert _{H^{1}}^{2}  \notag \\
&&-\frac{\lambda (4-p)}{4p}\int_{\mathbb{R}^{3}}K\left( x\right) \phi
_{K,s_{\lambda }^{\ast }v_{\lambda }-t_{n}^{\ast }w_{n}}(s_{\lambda }^{\ast
}v_{\lambda }-t_{n}^{\ast }w_{n})^{2}dx  \notag \\
&<&\lambda \left( \frac{4-p}{2p}-\frac{4-p}{4p}\right) \int_{\mathbb{R}%
^{3}}K\left( x\right) \phi _{K,s_{\lambda }^{\ast }v_{\lambda }-t_{n}^{\ast
}w_{n}}(s_{\lambda }^{\ast }v_{\lambda }-t_{n}^{\ast }w_{n})^{2}dx  \notag \\
&=&\frac{\lambda (4-p)}{4p}\int_{\mathbb{R}^{3}}K(x)\phi _{K,s_{\lambda
}^{\ast }v_{\lambda }-t_{n}^{\ast }w_{n}}(s_{\lambda }^{\ast }v_{\lambda
}-t_{n}^{\ast }w_{n})^{2}dx  \notag \\
&\leq &\frac{\lambda (4-p)}{4p}\overline{S}^{-2}S_{12/5}^{-4}K_{\max
}\left\Vert s_{\lambda }^{\ast }v_{\lambda }-t_{n}^{\ast }w_{n}\right\Vert
_{H^{1}}^{4}.  \label{5-21}
\end{eqnarray}%
Moreover, using Lemmas $\ref{h-2}$ and $\ref{h-4}$, leads to%
\begin{equation}
I_{\lambda }\left( s_{\lambda }^{\ast }v_{\lambda }-t_{n}^{\ast
}w_{n}\right) \geq I_{\lambda }\left( v_{\lambda }-w_{n}\right) =\alpha
_{\lambda }^{-}+\alpha _{\lambda }^{\infty }+o(1)  \label{5-22}
\end{equation}%
and%
\begin{equation}
\left\Vert s_{\lambda }^{\ast }v_{\lambda }-t_{n}^{\ast }w_{n}\right\Vert
_{H^{1}}^{2}=\left( s_{\lambda }^{\ast }\right) ^{2}\left\Vert v_{\lambda
}\right\Vert _{H^{1}}^{2}+\left( t_{n}^{\ast }\right) ^{2}\left\Vert
w_{n}\right\Vert _{H^{1}}^{2}-2s_{\lambda }^{\ast }t_{n}^{\ast }\left\langle
v_{\lambda },w_{n}\right\rangle \leq C_{0}.  \label{5-23}
\end{equation}%
Thus, by $(\ref{5-21})-(\ref{5-23}),$ we can conclude that for $\lambda >0$
sufficiently small and $n\geq n^{\ast },$ there holds%
\begin{eqnarray*}
\alpha _{\lambda }^{-}+\alpha _{\lambda }^{\infty }+o\left( 1\right)
&<&I_{\lambda }\left( s_{\lambda }^{\ast }v_{\lambda }-t_{n}^{\ast
}w_{n}\right)  \\
&<&\frac{\lambda (4-p)}{4p}\overline{S}^{-2}S_{12/5}^{-4}K_{\max }\left\Vert
s_{\lambda }^{\ast }v_{\lambda }-t_{n}^{\ast }w_{n}\right\Vert _{H^{1}}^{4}
\\
&\leq &\frac{\lambda (4-p)}{4p}\overline{S}^{-2}S_{12/5}^{-4}K_{\max
}C_{0}^{2} \\
&<&\alpha _{\lambda }^{\infty },
\end{eqnarray*}%
which a contradiction. This indicates that there exists a positive number $%
\lambda ^{\ast }\leq \min \{\lambda _{3},\lambda _{4}\}$ such that $%
s_{\lambda }^{\ast }v_{\lambda }-t_{n}^{\ast }w_{n}\in \mathbf{M}_{\lambda
}^{(1)}$ for all $0<\lambda <\lambda ^{\ast }$ and $n\geq n^{\ast }.$
Combining $(\ref{5-17})$ and $(\ref{5-18})$ gives $s_{\lambda }^{\ast
}v_{\lambda }-t_{n}^{\ast }w_{n}\in \mathbf{N}_{\lambda }^{(1)}.$ This
completes the proof.
\end{proof}

Let $w_{0}$ be the unique positive solution of the following Schr\"{o}dinger
equation%
\begin{equation}
\begin{array}{ll}
-\Delta u+u=f_{\infty }\left\vert u\right\vert ^{p-2}u & \text{ in }\mathbb{R%
}^{3}.%
\end{array}
\tag*{$\left( E_{0}^{\infty }\right) $}
\end{equation}%
From \cite{K}, one can see that%
\begin{equation*}
I_{0}^{\infty }\left( w_{0}\right) =\alpha _{0}^{\infty }:=\frac{p-2}{2p}%
\left( \frac{S_{p}^{p}}{f_{\infty }}\right) ^{\frac{2}{p-2}},
\end{equation*}%
where $I_{0}^{\infty }$ is the energy functional of Eq. $(E_{0}^{\infty })$
in $H^{1}(\mathbb{R}^{3})$ in the form
\begin{equation*}
I_{0}^{\infty }\left( u\right) =\frac{1}{2}\left\Vert u\right\Vert
_{H^{1}}^{2}-\frac{1}{p}\int_{\mathbb{R}^{3}}f_{\infty }\left\vert
u\right\vert ^{p}dx.
\end{equation*}%
Moreover, by \cite{GNN}, for any $\varepsilon >0,$ there exist positive
numbers $A_{\varepsilon }$ and $B_{0}$ such that
\begin{equation}
A_{\varepsilon }\exp \left( -\left( 1+\varepsilon \right) \left\vert
x\right\vert \right) \leq w_{0}\leq B_{0}\exp \left( -\left\vert
x\right\vert \right) \text{ for all }x\in \mathbb{R}^{N}.  \label{45}
\end{equation}%
For $n\in \mathbb{N},$ we define the sequence
\begin{equation*}
\overline{w}_{n}(x)=w_{0}(x-ne_{1}).
\end{equation*}%
Clearly, $I_{0}^{\infty }\left( \overline{w}_{n}\right) =I_{0}^{\infty
}\left( w_{0}\right) $ for all $n\in \mathbb{N},$ where $I_{0}^{\infty }$ is
the energy functional of Eq. $(E_{0}^{\infty }).$ Moreover, by $(\ref{45})$
one has%
\begin{equation*}
\overline{w}_{n}\left( x\right) =w_{0}\left( x-ne_{1}\right) \leq
C_{\varepsilon }\exp (\left\vert x\right\vert -n).
\end{equation*}%
Note that conditions $(F1),(F2),(K1)$ and $(K3)$ satisfy conditions $%
(D1),(D2)$ and $(D4)$ in \cite[Theorem 1.5]{SWF1}. Then from \cite[Theorem
1.5]{SWF1}, we obtain that there exists $\overline{\Lambda }>0$ such that
for every $0<\lambda <\overline{\Lambda },$ system $\left( SP_{\lambda
}\right) $ admits a positive solution $(\overline{v}_{\lambda },\phi _{K,%
\overline{v}_{\lambda }})\in H^{1}(\mathbb{R}^{3})\times D^{1,2}(\mathbb{R}%
^{3})$ satisfying%
\begin{equation*}
\frac{p-2}{4p}\left( \frac{S_{p}^{p}}{f_{\max }}\right) ^{\frac{2}{p-2}%
}<I_{\lambda }\left( \overline{v}_{\lambda }\right) =\alpha _{\lambda
}^{-}<\alpha _{0}^{\infty },
\end{equation*}%
and $\overline{v}_{\lambda }$ has also exponential decay like $(\ref{3-1}).$
Moreover, similar to Lemma \ref{h-1} and Proposition \ref{h-5}, we have the
following two conclusions.

\begin{lemma}
\label{h-7}Suppose that conditions ${(F1)},{(F2)},\left( K1\right) $ and $%
\left( K3\right) $ hold. Then for each $0<\varepsilon <1$ there exists $%
C_{\varepsilon }>0$ such that\newline
$\left( i\right) $ $\int_{\mathbb{R}^{3}}K(x)\phi _{K,\overline{w}_{n}}%
\overline{v}_{\lambda }^{2}dx=\int_{\mathbb{R}^{3}}K(x)\phi _{K,\overline{v}%
_{\lambda }}\overline{w}_{n}^{2}dx\leq C_{\varepsilon }e^{-n^{1-\varepsilon
}};$\newline
$\left( ii\right) $ $\left\vert \int_{\mathbb{R}^{3}}K(x)\phi _{K,\overline{v%
}_{\lambda }-\overline{w}_{n}}\left( \overline{v}_{\lambda }-\overline{w}%
_{n}\right) ^{2}dx-\int_{\mathbb{R}^{3}}K(x)\phi _{K,\overline{v}_{\lambda }}%
\overline{v}_{\lambda }^{2}dx-\int_{\mathbb{R}^{3}}K(x)\phi _{K,\overline{w}%
_{n}}\overline{w}_{n}^{2}dx\right\vert \leq C_{\varepsilon
}e^{-n^{1-\varepsilon }}.$
\end{lemma}

\begin{proposition}
\label{h-6}Suppose that $2<p<4,$ and conditions ${(F1)},{(F2)},\left(
K1\right) $ and $\left( K3\right) $ hold. Then there exist two positive
numbers $\overline{\lambda }^{\ast }\leq \min \{\widetilde{\lambda },%
\overline{\Lambda }\}$ and $\overline{n}^{\ast }\in \mathbb{N}$ such that
for every $0<\lambda <\overline{\lambda }^{\ast }$ and $n\geq \overline{n}%
^{\ast },$ there exists $\left( \overline{s}_{\lambda }^{\ast },\overline{t}%
_{n}^{\ast }\right) \in \left( 0,\infty \right) \times \left( 0,\infty
\right) $ such that $\overline{s}_{\lambda }^{\ast }\overline{v}_{\lambda }-%
\overline{t}_{n}^{\ast }\overline{w}_{n}\in \mathbf{N}_{\lambda }^{(1)}$ and%
\begin{equation*}
I_{\lambda }\left( \overline{s}_{\lambda }^{\ast }\overline{v}_{\lambda }-%
\overline{t}_{n}^{\ast }\overline{w}_{n}\right) <\alpha _{\lambda
}^{-}+\alpha _{0}^{\infty }.
\end{equation*}
\end{proposition}

The proofs of the two results above are analogous to those of Lemma \ref{h-1}
and Proposition \ref{h-5}, respectively, and so we omit here.

\section{Palais--Smale Sequences}

Define%
\begin{equation*}
\theta _{\lambda }^{-}=\inf_{u\in \mathbf{N}_{\lambda }^{(1)}}I_{\lambda
}\left( u\right) .
\end{equation*}%
Then by Lemma \ref{h3-5} and Proposition \ref{h-5} or \ref{h-6}, we have%
\begin{equation}
2\alpha _{\lambda }^{-}\leq \theta _{\lambda }^{-}<\alpha _{\lambda
}^{-}+\alpha _{\lambda }^{\infty }  \label{18-0}
\end{equation}%
or
\begin{equation*}
2\alpha _{\lambda }^{-}\leq \theta _{\lambda }^{-}<\alpha _{\lambda
}^{-}+\alpha _{0}^{\infty }.
\end{equation*}%
Next, we define%
\begin{equation*}
\Phi _{\lambda }^{+}\left( u\right) =\frac{\left\Vert u^{+}\right\Vert
_{H^{1}}^{2}+\lambda \left( \int_{\mathbb{R}^{3}}K(x)\phi
_{K,u^{-}}(u^{+})^{2}dx+\int_{\mathbb{R}^{3}}K(x)\phi
_{K,u^{+}}(u^{+})^{2}dx\right) }{\int_{\mathbb{R}^{3}}f(x)\left\vert
u^{+}\right\vert ^{p}dx}
\end{equation*}%
and
\begin{equation*}
\Phi _{\lambda }^{-}\left( u\right) =\frac{\left\Vert u^{-}\right\Vert
_{H^{1}}^{2}+\lambda \left( \int_{\mathbb{R}^{3}}K(x)\phi
_{K,u^{+}}(u^{-})^{2}dx+\int_{\mathbb{R}^{3}}K(x)\phi
_{K,u^{-}}(u^{-})^{2}dx\right) }{\int_{\mathbb{R}^{3}}f(x)\left\vert
u^{-}\right\vert ^{p}dx}.
\end{equation*}%
Then for each $u\in \mathbf{N}_{\lambda }^{(1)},$ there holds $\Phi
_{\lambda }^{+}\left( u\right) =\Phi _{\lambda }^{-}\left( u\right) =1.$
Furthermore, we have the following results.

\begin{lemma}
\label{f4-1}For each $\epsilon >0$ there exists $\mu (\epsilon )>0$ such
that for every $v\in \mathbf{N}_{\lambda }^{(1)}$ and $u\in H^{1}(\mathbb{R}%
^{3})$ with $\left\Vert v-u\right\Vert _{H^{1}}<\mu (\epsilon ),$ there
holds $\left\vert \Phi _{\lambda }^{+}\left( u\right) -1\right\vert
+\left\vert \Phi _{\lambda }^{-}\left( u\right) -1\right\vert <\epsilon .$
\end{lemma}

\begin{lemma}
\label{f4-2}Suppose that $2<p<4.$ Then for each $v_{0}\in \mathbf{N}%
_{\lambda }^{(1)},$ there exists a map $\phi _{\lambda }:H^{1}(\mathbb{R}%
^{3})\rightarrow \mathbb{R}^{2}$ such that\newline
$\left( i\right) \ \phi _{\lambda }\left(
s_{1}v_{0}^{+}+s_{2}v_{0}^{-}\right) =\left( s_{1},s_{2}\right) $ for $%
\left( s_{1},s_{2}\right) \in \left[ 0,\widetilde{s}_{\lambda }\right]
\times \left[ 0,\widetilde{t}_{\lambda }\right] ;$\newline
$\left( ii\right) \ \phi _{\lambda }\left( u\right) =\left( 1,1\right) $ if
and only if $u\in \mathbf{N}_{\lambda }^{(1)}.$
\end{lemma}

The proofs of Lemmas \ref{f4-1} and \ref{f4-2} are almost the same as those
in Clapp and Weth \cite[Lemma 13]{CW2} and we omit them here.

\begin{proposition}
\label{f5-1}Let $\epsilon ,\mu (\epsilon )>0$ be as in Lemma \ref{f4-1}.
Then for each
\begin{equation*}
0<\eta <C\left( p\right) \left( \frac{S_{p}^{p}}{f_{\infty }}\right) ^{\frac{%
2}{p-2}}-\theta _{\lambda }^{-}
\end{equation*}%
and $\mu \in (0,\mu (\epsilon )),$ there exists $u_{0}\in H^{1}(\mathbb{R}%
^{3})$ such that for every $0<\lambda <\widetilde{\lambda },$\newline
$\left( i\right) \ dist\left( u_{0},\mathbf{N}_{\lambda }^{(1)}\right) \leq
\mu ;$\newline
$\left( ii\right) \ I_{\lambda }\left( u_{0}\right) \in \lbrack \theta
_{\lambda }^{-},\theta _{\lambda }^{-}+\eta );$\newline
$\left( iii\right) \ \left\Vert I_{\lambda }^{\prime }\left( u_{0}\right)
\right\Vert _{H^{-1}}\leq \max \left\{ \sqrt{\eta },\frac{\eta }{\mu }%
\right\} ;$\newline
$\left( iv\right) \ \left\vert \Phi _{\lambda }^{+}\left( u\right)
-1\right\vert +\left\vert \Phi _{\lambda }^{-}\left( u\right) -1\right\vert
<\epsilon .$
\end{proposition}

\begin{proof}
Let us fix $v_{0}\in \mathbf{N}_{\lambda }^{(1)}$ such that $I_{\lambda
}\left( v_{0}\right) <\theta _{\lambda }^{-}+\eta ,$ and fix $\widetilde{s}%
_{\lambda },\widetilde{t}_{\lambda }>1$ as in Lemma \ref{h3-3} such that $%
I_{\lambda }\left( \widetilde{s}_{\lambda }v_{0}^{+}+\widetilde{t}_{\lambda
}v_{0}^{-}\right) \leq 0.$ Let $\phi _{\lambda }:H^{1}(\mathbb{R}%
^{3})\rightarrow \mathbb{R}^{2}$ as in Lemma \ref{f4-2}. We define a map $%
\beta _{\lambda }:Q_{\lambda }\rightarrow H^{1}(\mathbb{R}^{3})$ by%
\begin{equation*}
\beta _{\lambda }\left( s_{1},s_{2}\right) =s_{1}v_{0}^{+}+s_{2}v_{0}^{-},
\end{equation*}%
where $Q_{\lambda }=\left[ 0,\widetilde{s}_{\lambda }\right] \times \left[ 0,%
\widetilde{t}_{\lambda }\right] .$ Then $\phi _{\lambda }\circ \beta
_{\lambda }=id:Q_{\lambda }\rightarrow Q_{\lambda }.$ In particular, there
holds%
\begin{equation}
\deg \left( \phi _{\lambda }\circ \beta _{\lambda },Q_{\lambda },\left(
1,1\right) \right) =1.  \label{20}
\end{equation}%
Moreover, we also have
\begin{equation}
I_{\lambda }\left( \beta _{\lambda }\left( s_{1},s_{2}\right) \right) \leq
I_{\lambda }\left( v_{0}\right) <\theta _{\lambda }^{-}+\eta \text{ for all }%
\left( s_{1},s_{2}\right) \in Q_{\lambda }.  \label{21}
\end{equation}%
Now we choose a Lipschitz continuous function $\chi :\mathbb{R}\rightarrow
\mathbb{R}$ such that $0\leq \chi \leq 1,\chi \left( s\right) =1$ for $s\geq
0$ and $\chi \left( s\right) =0$ for $s\leq -1.$ Since $I_{\lambda }\in
C^{2}(H^{1}(\mathbb{R}^{3}),\mathbb{R)},$ there is a semiflow $\varphi
:[0,\infty )\times H^{1}(\mathbb{R}^{3})\rightarrow H^{1}(\mathbb{R}^{3})$
satisfying%
\begin{equation*}
\left\{
\begin{array}{l}
\frac{\partial }{\partial t}\varphi \left( t,u\right) =-\chi \left(
I_{\lambda }\left( \varphi \left( t,u\right) \right) \right) I_{\lambda
}^{\prime }\left( \varphi \left( t,u\right) \right) , \\
\varphi \left( 0,u\right) =u.%
\end{array}%
\right.
\end{equation*}%
For convenience, we always write $\varphi \left( t,\cdot \right) $ by $%
\varphi ^{t}$ in the sequel. Since $\max \left\{ I_{\lambda }\left(
\widetilde{s}_{\lambda }v_{0}^{+}\right) ,I_{\lambda }\left( \widetilde{t}%
_{\lambda }v_{0}^{-}\right) \right\} <0,$ similar to the argument in Lemma %
\ref{h3-3}, we have
\begin{equation*}
\sup I_{\lambda }\left( \beta _{\lambda }\left( \partial Q_{\lambda }\right)
\right) <2\alpha _{\lambda }^{-}.
\end{equation*}%
Hence,%
\begin{equation*}
\left( \varphi ^{t}\circ \beta _{\lambda }\right) \left( \partial Q_{\lambda
}\right) \cap \mathbf{N}_{\lambda }^{(1)}=\emptyset \text{ for all }t\geq 0.
\end{equation*}%
Using Lemma \ref{f4-2} gives%
\begin{equation*}
\left( \phi _{\lambda }\circ \varphi ^{t}\circ \beta _{\lambda }\right)
\left( y\right) \neq \left( 1,1\right) \text{ for all }y\in \partial
Q_{\lambda }\text{ and }t\geq 0.
\end{equation*}%
By $(\ref{20})$ and the global continuation principle of Leray-Schauder (see
e.g. Zeider \cite[p.629]{Za}), we obtain that there exists a connected
subset $Z\subset Q_{\lambda }\times \left[ 0,1\right] $ such that
\begin{equation*}
\begin{array}{l}
\left( 1,1,0\right) \in Z, \\
\varphi ^{t}\left( \beta _{\lambda }\left( s_{1},s_{2}\right) \right) \in
\mathbf{N}_{\lambda }^{(1)}\text{ for all }\left( s_{1},s_{2},t\right) \in Z,
\\
Z\cap \left( Q_{\lambda }\times \left\{ 1\right\} \right) \neq \emptyset .%
\end{array}%
\end{equation*}%
Set
\begin{equation*}
\Gamma =\left\{ \varphi ^{t}\left( \beta _{\lambda }\left(
s_{1},s_{2}\right) \right) \in \mathbf{N}_{\lambda }^{(1)}:\left(
s_{1},s_{2},t\right) \in Z\right\} .
\end{equation*}%
From $(\ref{21})$ it follows that%
\begin{equation*}
\sup_{u\in \Gamma }I_{\lambda }\left( u\right) <\theta _{\lambda }^{-}+\eta ,
\end{equation*}%
which implies that $\Gamma \subset \mathbf{N}_{\lambda }^{(1)},$ since $Z$
is connected. Now we pick $\left( \bar{s}_{1},\bar{s}_{2},1\right) \in Z\cap
\left( Q_{\lambda }\times \left\{ 1\right\} \right) $ and set
\begin{equation*}
v_{1}:=\psi _{\lambda }\left( \bar{s}_{1},\bar{s}_{2}\right) \text{ and }%
v_{2}:=\varphi ^{1}(v_{1}).
\end{equation*}%
Clearly, $v_{2}\in \Gamma \subset \mathbf{N}_{\lambda }^{(1)}$ and $\Phi
_{\lambda }^{+}\left( v_{2}\right) =\Phi _{\lambda }^{-}\left( v_{2}\right)
=1.$ We distinguish two cases as follows:\newline
Case $(i):\left\Vert \varphi ^{t}\left( v_{1}\right) -v_{2}\right\Vert
_{H^{1}}\leq \mu $ for all $t\in \left[ 0,1\right] .$ By Lemma \ref{f4-1}
one has%
\begin{equation*}
\left\vert \Phi _{\lambda }^{+}\left( \varphi ^{t}\left( v_{1}\right)
\right) -1\right\vert +\left\vert \Phi _{\lambda }^{-}\left( \varphi
^{t}\left( v_{1}\right) \right) -1\right\vert <\epsilon \text{ for all }t\in %
\left[ 0,1\right] .
\end{equation*}%
Choosing $t_{0}\in \left[ 0,1\right] $ with%
\begin{equation*}
\left\Vert I_{\lambda }^{\prime }\left( \varphi ^{t_{0}}\left( v_{1}\right)
\right) \right\Vert _{H^{-1}}=\min_{0\leq t\leq 1}\left\Vert I_{\lambda
}^{\prime }\left( \varphi ^{t}\left( v_{1}\right) \right) \right\Vert
_{H^{-1}}
\end{equation*}%
and setting $u_{0}=\varphi ^{t_{0}}(v_{1}).$ Then, we have%
\begin{eqnarray*}
\eta  &\geq &I_{\lambda }\left( v_{1}\right) -I_{\lambda }\left(
v_{2}\right) =-\int_{0}^{1}\frac{\partial }{\partial t}I_{\lambda }\left(
\varphi ^{t}\left( v_{1}\right) \right) dt \\
&=&\int_{0}^{1}\left\Vert I_{\lambda }^{\prime }\left( \varphi
^{t_{0}}\left( v_{1}\right) \right) \right\Vert _{H^{-1}}^{2}dt\geq
\left\Vert I_{\lambda }^{\prime }\left( u_{0}\right) \right\Vert
_{H^{-1}}^{2}.
\end{eqnarray*}%
Therefore, $u_{0}$ satisfies the desired properties.\newline
Case $(ii):$ There exists $\bar{t}\in \left[ 0,1\right] $ such that $%
\left\Vert \varphi ^{\bar{t}}\left( v_{1}\right) -v_{2}\right\Vert
_{H^{1}}>\mu .$ Let%
\begin{equation*}
t_{1}=\sup \left\{ t\geq \bar{t}\ |\ \left\Vert \varphi ^{t}\left(
v_{1}\right) -v_{2}\right\Vert >\mu \right\} .
\end{equation*}%
Then by Lemma \ref{f4-1}, we have%
\begin{equation*}
\left\vert \Phi _{\lambda }^{+}\left( \varphi ^{t}\left( v_{1}\right)
\right) -1\right\vert +\left\vert \Phi _{\lambda }^{-}\left( \varphi
^{t}\left( v_{1}\right) \right) -1\right\vert <\epsilon
\end{equation*}%
for all $t\in \lbrack t_{1},1].$ Choosing $t_{0}\in \left[ t_{1},1\right] $
with
\begin{equation*}
\left\Vert I_{\lambda }^{\prime }\left( \varphi ^{t_{0}}\left( v_{1}\right)
\right) \right\Vert _{H^{-1}}=\min_{t_{1}\leq t\leq 1}\left\Vert I_{\lambda
}^{\prime }\left( \varphi ^{t}\left( v_{1}\right) \right) \right\Vert
_{H^{-1}}
\end{equation*}%
and setting $u_{0}=\varphi ^{t_{0}}(v_{1}).$ Then there holds%
\begin{equation*}
\mu \leq \int_{t_{1}}^{1}\left\Vert \frac{\partial }{\partial t}\varphi
^{t}\left( v_{1}\right) \right\Vert _{H^{1}}dt\leq
\int_{t_{1}}^{1}\left\Vert I_{\lambda }^{\prime }\left( \varphi ^{t}\left(
v_{1}\right) \right) \right\Vert _{H^{-1}}dt
\end{equation*}%
and%
\begin{eqnarray*}
\eta  &\geq &I_{\lambda }\left( \varphi ^{t_{1}}\left( v_{1}\right) \right)
-I_{\lambda }\left( v_{2}\right) =\int_{t_{1}}^{1}\left\Vert I_{\lambda
}^{\prime }\left( \varphi ^{t}v_{1}\right) \right\Vert _{H^{-1}}^{2}dt \\
&\geq &\left\Vert I_{\lambda }^{\prime }\left( u_{0}\right) \right\Vert
_{H^{-1}}\int_{t_{1}}^{1}\left\Vert I_{\lambda }^{\prime }\left( \varphi
^{t}\left( v_{1}\right) \right) \right\Vert _{H^{-1}}dt,
\end{eqnarray*}%
which implies that $\left\Vert I_{\lambda }^{\prime }\left( u_{0}\right)
\right\Vert _{H^{-1}}\leq \frac{\eta }{\mu }.$ Therefore, $u_{0}$ satisfies
the desired properties. The proof is complete.
\end{proof}

\begin{corollary}
\label{f5-2}For each $0<\lambda <\widetilde{\lambda },$ there exists a
sequence $\left\{ u_{n}\right\} \subset H^{1}(\mathbb{R}^{3})$ such that%
\newline
$\left( i\right) \ dist\left( u_{n},\mathbf{N}_{\lambda }^{(1)}\right)
\rightarrow 0;$\newline
$\left( ii\right) \ I_{\lambda }\left( u_{n}\right) \rightarrow \theta
_{\lambda }^{-};$\newline
$\left( iii\right) \ I_{\lambda }^{\prime }(u_{n})=o(1)$ strongly in $H^{-1}(%
\mathbb{R}^{3});$\newline
$\left( iv\right) $ $\left\vert \Phi _{\lambda }^{+}\left( u_{n}\right)
-1\right\vert +\left\vert \Phi _{\lambda }^{-}\left( u_{n}\right)
-1\right\vert \rightarrow 0.$
\end{corollary}

\section{Proof of Theorem \protect\ref{t2}}

Before proving Theorem \ref{t2}, we first give a precise description of the
Palais--Smale sequence for $I_{\lambda }$ in this section.

\begin{proposition}
\label{d1}Suppose that $2<p<4,$ and conditions $(F1)-\left( F2\right) $ and $%
(K1)-\left( K2\right) $ hold. Let $\left\{ u_{n}\right\} \subset H^{1}(%
\mathbb{R}^{3})$ be a sequence satisfying\newline
$\left( i\right) \ dist\left( u_{n},\mathbf{N}_{\lambda }^{(1)}\right)
\rightarrow 0;$\newline
$\left( ii\right) \ I_{\lambda }\left( u_{n}\right) \rightarrow \theta
_{\lambda }^{-};$\newline
$\left( iii\right) \ I_{\lambda }^{\prime }(u_{n})=o(1)$ strongly in $H^{-1}(%
\mathbb{R}^{3});$\newline
$\left( iv\right) $ $\left\vert \Phi _{\lambda }^{+}\left( u_{n}\right)
-1\right\vert +\left\vert \Phi _{\lambda }^{-}\left( u_{n}\right)
-1\right\vert \rightarrow 0.$\newline
Then there exist a subsequence $\left\{ u_{n}\right\} $ and $u_{\lambda }\in
\mathbf{N}_{\lambda }^{(1)}$ such that $u_{n}\rightarrow u_{\lambda }$
strongly in $H^{1}(\mathbb{R}^{3})$ for each $0<\lambda <\lambda ^{\ast }.$
\end{proposition}

\begin{proof}
Since $\left\{ u_{n}\right\} $ is bounded in $H^{1}(\mathbb{R}^{3}),$ we can
assume that there exists $u_{\lambda }\in H^{1}(\mathbb{R}^{3})$ such that%
\begin{eqnarray}
u_{n} &\rightharpoonup &u_{\lambda }\text{ and }u_{n}^{\pm }\rightharpoonup
u_{\lambda }^{\pm }\text{ weakly in }H^{1}(\mathbb{R}^{3}),  \notag \\
u_{n} &\rightarrow &u_{\lambda }\text{ and }u_{n}^{\pm }\rightarrow
u_{\lambda }^{\pm }\text{ strongly in }L_{loc}^{r}(\mathbb{R}^{3})\text{ for
}1\leq r<2^{\ast },  \label{18-2} \\
u_{n} &\rightarrow &u_{\lambda }\text{ and }u_{n}^{\pm }\rightarrow
u_{\lambda }^{\pm }\text{ a.e. in }\mathbb{R}^{3}.  \notag
\end{eqnarray}

First, we claim that$\ u_{\lambda }^{\pm }\not\equiv 0.$ Suppose on the
contrary. Then we can assume without loss of generality that $u_{\lambda
}^{+}\equiv 0.$ Since$\ dist\left( u_{n},\mathbf{N}_{\lambda }^{(1)}\right)
\rightarrow 0$ as $n\rightarrow \infty $ and $2\alpha _{\lambda }^{-}\leq
\theta _{\lambda }^{-}<\alpha _{\lambda }^{\infty }+\alpha _{\lambda }^{-},$
we deduce from the Sobolev imbedding theorem that $\left\Vert
u_{n}^{+}\right\Vert _{H^{1}}>\nu >0$ for some constant $\nu $ and for all $%
n>0.$ Applying the concentration-compactness principle of P.L. Lions \cite%
{Li}, there exist positive constants $R,d$ and a sequence $\left\{
x_{n}\right\} \subset \mathbb{R}^{3}$ such that%
\begin{equation}
\int_{B_{R}\left( 0\right) }\left\vert u_{n}^{+}\left( x+x_{n}\right)
\right\vert ^{p}dx\geq d\text{ for }n\text{ sufficiently large.}  \label{23}
\end{equation}%
We will show that $\left\{ x_{n}\right\} $ is a unbounded sequence in $%
\mathbb{R}^{3}.$ Suppose otherwise, we can assume that $x_{n}\rightarrow
x_{0}$ for some $x_{0}\in \mathbb{R}^{3}.$ It follows from $(\ref{18-2})$
and $(\ref{23})$ that%
\begin{equation*}
\int_{B_{R}\left( x_{0}\right) }\left\vert u_{\lambda }^{+}\right\vert
^{p}dx\geq d,
\end{equation*}%
which contradicts with $u_{\lambda }^{+}\equiv 0.$ Thus, $\left\{
x_{n}\right\} $ is a unbounded sequence in $\mathbb{R}^{3}.$ Set $\widetilde{%
u}_{n}\left( x\right) =u_{n}\left( x+x_{n}\right) .$ Clearly, $\left\{
\widetilde{u}_{n}\right\} $ is also bounded in $H^{1}(\mathbb{R}^{3}).$ Then
we may assume that there exists $\widetilde{u}_{0}\in H^{1}(\mathbb{R}^{3})$
such that
\begin{eqnarray}
\widetilde{u}_{n} &\rightharpoonup &\widetilde{u}_{\lambda }\text{ and }%
\widetilde{u}_{n}^{\pm }\rightharpoonup \widetilde{u}_{\lambda }^{\pm }\text{%
weakly in }H^{1}(\mathbb{R}^{3}),  \label{24} \\
\widetilde{u}_{n} &\rightarrow &\widetilde{u}_{\lambda }\text{ and }%
\widetilde{u}_{n}^{\pm }\rightarrow \widetilde{u}_{\lambda }^{\pm }\text{
a.e. in }\mathbb{R}^{3}.  \notag
\end{eqnarray}%
By $(\ref{23}),$ we have $\widetilde{u}_{\lambda }^{+}\not\equiv 0$ in $%
\mathbb{R}^{3}.$ Note that
\begin{equation}
\left\Vert \widetilde{u}_{n}^{\pm }\right\Vert _{H^{1}}=\left\Vert
u_{n}^{\pm }\right\Vert _{H^{1}}<D_{1}<\left( \frac{2S_{p}^{p}}{f_{\max
}\left( 4-p\right) }\right) ^{\frac{1}{p-2}},  \label{18-8}
\end{equation}%
it follows from Fatou's Lemma that
\begin{equation*}
\left\Vert \widetilde{u}_{\lambda }^{+}\right\Vert _{H^{1}}\leq \lim
\inf_{n\rightarrow \infty }\left\Vert \widetilde{u}_{n}^{+}\right\Vert
_{H^{1}}\leq D_{1}<\left( \frac{2S_{p}^{p}}{f_{\max }\left( 4-p\right) }%
\right) ^{\frac{1}{p-2}}.
\end{equation*}%
By conditions $(F1),(K1)$ and$(K2),$ we have $K\left( x-x_{n}\right)
\rightarrow K_{\infty }$ and $f\left( x-x_{n}\right) \rightarrow f_{\infty }$
as $n\rightarrow \infty .$ Thus, from Lemma \ref{h2} and the fact of $%
I_{\lambda }^{\prime }(u_{n})\rightarrow 0$ on $H^{-1}(\mathbb{R}^{3})$ it
follows that%
\begin{eqnarray}
&&\left\Vert \widetilde{u}_{\lambda }^{+}\right\Vert _{H^{1}}^{2}+\lambda
\left( \int_{\mathbb{R}^{3}}K_{\infty }\phi _{K_{\infty },\widetilde{u}%
_{\lambda }^{+}}(\widetilde{u}_{\lambda }^{+})^{2}dx+\int_{\mathbb{R}%
^{3}}K_{\infty }\phi _{K_{\infty },\widetilde{u}_{\lambda }^{-}}(\widetilde{u%
}_{\lambda }^{+})^{2}dx\right)   \notag \\
&=&\int_{\mathbb{R}^{3}}f_{\infty }\left\vert \widetilde{u}_{\lambda
}^{+}\right\vert ^{p}dx,  \label{18-6}
\end{eqnarray}%
and%
\begin{eqnarray}
&&\left\Vert \widetilde{u}_{n}^{\pm }\right\Vert _{H^{1}}^{2}+\lambda \left(
\int_{\mathbb{R}^{3}}K_{\infty }\phi _{K_{\infty },\widetilde{u}_{n}^{\pm }}(%
\widetilde{u}_{n}^{\pm })^{2}dx+\int_{\mathbb{R}^{3}}K_{\infty }\phi
_{K_{\infty },\widetilde{u}_{n}^{\mp }}^{2}(\widetilde{u}_{n}^{\pm
})^{2}dx\right)   \notag \\
&=&\int_{\mathbb{R}^{3}}f_{\infty }\left\vert \widetilde{u}_{n}^{\pm
}\right\vert ^{p}dx+o(1).  \label{18-4}
\end{eqnarray}%
Set $v_{n}=\widetilde{u}_{n}^{+}-\widetilde{u}_{\lambda }^{+}.$ We
distinguish two cases as follows:\newline
Case $I:\left\Vert v_{n}\right\Vert _{H^{1}}\rightarrow 0$ as $n\rightarrow
\infty .$ Since $dist\left( u_{n},\mathbf{N}_{\lambda }^{(1)}\right)
\rightarrow 0,$ it follows from $(\ref{18-6})$ and Lemma \ref{h3-5} that%
\begin{eqnarray*}
I_{\lambda }(u_{n}) &=&J_{\lambda }^{+}(u_{n}^{+},u_{n}^{-})+J_{\lambda
}^{-}(u_{n}^{+},u_{n}^{-}) \\
&=&J_{\lambda }^{+}\left( \widetilde{u}_{n}^{+},\widetilde{u}_{n}^{-}\right)
+J_{\lambda }^{-}\left( u_{n}^{+},u_{n}^{-}\right)  \\
&=&(J_{\lambda }^{+})^{\infty }\left( \widetilde{u}_{\lambda }^{+},%
\widetilde{u}_{\lambda }^{-}\right) +J_{\lambda }^{-}\left(
u_{n}^{+},u_{n}^{-}\right) +o\left( 1\right)  \\
&\geq &\alpha _{\lambda }^{\infty }+\alpha _{\lambda }^{-}+o(1),
\end{eqnarray*}%
where $(J_{\lambda }^{+})^{\infty }=J_{\lambda }^{+}$ with $K(x)\equiv
K_{\infty }$ and $f(x)\equiv f_{\infty }$. Thus, $\theta _{\lambda }^{-}\geq
\alpha _{\lambda }^{\infty }+\alpha _{\lambda }^{-},$ which contradicts to $%
\theta _{\lambda }^{-}<\alpha _{\lambda }^{\infty }+\alpha _{\lambda }^{-}.$%
\newline
Case $II:\left\Vert v_{n}\right\Vert _{H^{1}}\geq c_{0}$ for large $n$ and
for some constant $c_{0}>0.$ Following Brezis-Lieb Lemma \cite{BLi} and \cite%
[Lemma 2.2]{ZZ}, together with $(\ref{18-6})$ and $(\ref{18-4}),$ we have%
\begin{equation}
\left\Vert v_{n}\right\Vert _{H^{1}}^{2}+\lambda \left( \int_{\mathbb{R}%
^{3}}K_{\infty }\phi _{K_{\infty },v_{n}}v_{n}^{2}dx+\int_{\mathbb{R}%
^{3}}K_{\infty }\phi _{K_{\infty },\widetilde{u}_{n}^{-}}v_{n}^{2}dx\right)
-\int_{\mathbb{R}^{3}}f_{\infty }\left\vert v_{n}\right\vert ^{p}dx=o(1).
\label{18-9}
\end{equation}%
Note that $\Vert \widetilde{u}_{\lambda }^{+}\Vert _{H^{1}}\geq \left( \frac{%
S_{p}^{p}}{f_{\max }}\right) ^{\frac{1}{p-2}}$ and $\left\Vert
v_{n}\right\Vert _{H^{1}}^{2}=\left\Vert \widetilde{u}_{n}^{+}\right\Vert
_{H^{1}}^{2}-\Vert \widetilde{u}_{\lambda }^{+}\Vert _{H^{1}}^{2}+o(1).$
Then it follows from $\left( \ref{24}\right) $ and $\left( \ref{18-8}\right)
$ that%
\begin{equation}
\left\Vert v_{n}\right\Vert _{H^{1}}<D_{1}<\left( \frac{2S_{p}^{p}}{f_{\max
}\left( 4-p\right) }\right) ^{\frac{1}{p-2}}\text{ for sufficiently large }n.
\label{18-10}
\end{equation}%
By $(\ref{18-9}),(\ref{18-10})$ and the fact of $\left\Vert v_{n}\right\Vert
_{H^{1}}\geq c_{0}$ for sufficiently large $n,$ it is straightforward to
find a sequence $\left\{ s_{n}\right\} \subset \mathbb{R}^{+}$ with $%
s_{n}\rightarrow 1$ as $n\rightarrow \infty $ such that
\begin{eqnarray*}
&&\left\Vert s_{n}v_{n}\right\Vert _{H^{1}}^{2}+\lambda \left( \int_{\mathbb{%
R}^{3}}K_{\infty }\phi _{K_{\infty },s_{n}v_{n}}(s_{n}v_{n})^{2}dx+\int_{%
\mathbb{R}^{3}}K_{\infty }\phi _{K_{\infty },\widetilde{u}%
_{n}^{-}}(s_{n}v_{n})^{2}dx\right)  \\
&=&\int_{\mathbb{R}^{3}}f_{\infty }\left\vert s_{n}v_{n}\right\vert ^{p}dx.
\end{eqnarray*}%
Thus, similar to the argument in Lemma \ref{h3-5}, we obtain%
\begin{eqnarray*}
&&\frac{1}{2}\left\Vert v_{n}\right\Vert _{H^{1}}^{2}+\frac{\lambda }{4}%
\left( \int_{\mathbb{R}^{3}}K_{\infty }\phi _{K_{\infty
},v_{n}}v_{n}^{2}dx+\int_{\mathbb{R}^{3}}K_{\infty }\phi _{K_{\infty },%
\widetilde{u}_{n}^{-}}v_{n}^{2}dx\right) -\frac{1}{p}\int_{\mathbb{R}%
^{3}}f_{\infty }\left\vert v_{n}\right\vert ^{p}dx \\
&\geq &\alpha _{\lambda }^{\infty }+o(1),
\end{eqnarray*}%
where we have used the fact of $s_{n}\rightarrow 1.$ It follows from Lemma %
\ref{h3-5}, Brezis-Lieb Lemma \cite{BLi} and \cite[Lemma 2.2]{ZZ} that
\begin{eqnarray*}
I_{\lambda }\left( u_{n}\right)  &=&J_{\lambda }^{+}\left(
u_{n}^{+},u_{n}^{-}\right) +J_{\lambda }^{-}\left(
u_{n}^{+},u_{n}^{-}\right)  \\
&=&(J_{\lambda }^{+})^{\infty }\left( \widetilde{u}_{n}^{+},\widetilde{u}%
_{n}^{-}\right) +J_{\lambda }^{-}\left( u_{n}^{+},u_{n}^{-}\right) +o\left(
1\right)  \\
&=&\frac{1}{2}\left\Vert v_{n}\right\Vert _{H^{1}}^{2}+\frac{\lambda }{4}%
\left( \int_{\mathbb{R}^{3}}K_{\infty }\phi _{K_{\infty
},v_{n}}v_{n}^{2}dx+\int_{\mathbb{R}^{3}}K_{\infty }\phi _{K_{\infty },%
\widetilde{u}_{n}^{-}}v_{n}^{2}dx\right)  \\
&&-\frac{1}{p}\int_{\mathbb{R}^{3}}f_{\infty }\left\vert v_{n}\right\vert
^{p}dx+(J_{\lambda }^{+})^{\infty }\left( \widetilde{u}_{\lambda }^{+},%
\widetilde{u}_{\lambda }^{-}\right) +J_{\lambda }^{-}\left(
u_{n}^{+},u_{n}^{-}\right) +o\left( 1\right)  \\
&\geq &2\alpha _{\lambda }^{\infty }+\alpha _{\lambda }^{-}+o(1),
\end{eqnarray*}%
which implies that
\begin{equation*}
\lim_{n\rightarrow \infty }I_{\lambda }(u_{n})=\theta _{\lambda }^{-}\geq
2\alpha _{\lambda }^{\infty }+\alpha _{\lambda }^{-}.
\end{equation*}%
This contradicts to $(\ref{18-0}).$ Hence, $u_{\lambda }^{+}\not\equiv 0.$
Similarly, we also obtain $u_{\lambda }^{-}\not\equiv 0.$

Next, we show that $u_{n}\rightarrow u_{0}$ strongly in $H^{1}(\mathbb{R}%
^{3}).$ Similar to the argument of Case $II,$ we can easily arrive at the
conclusion. Moreover, we have $u_{\lambda }\in \mathbf{N}_{\lambda }^{(1)}$
and $I_{\lambda }(u_{\lambda })=\theta _{\lambda }^{-}.$ This indicates that
$u_{\lambda }$ is a nodal solution for each $0<\lambda <\lambda ^{\ast }$.
The proof is complete.
\end{proof}

\textbf{We are ready to prove Theorem \ref{t2}:} By Corollary \ref{f5-2} and
Proposition \ref{d1}, for each $0<\lambda <\lambda ^{\ast },$ Eq. $\left(
E_{\lambda }\right) $ has a nodal solution $u_{\lambda }$ such that $%
I_{\lambda }\left( u_{\lambda }\right) =\theta _{\lambda }^{-}.$ Moreover,
similar to the argument in \cite[Theorem 1.3]{AS}, $u_{\lambda }$ changes
sign exactly once in $\mathbb{R}^{3}.$ Consequently, system $(SP_{\lambda })$
admits a nodal solution $(u_{\lambda },\phi _{K,u_{\lambda }})\in H^{1}(%
\mathbb{R}^{3})\times D^{1,2}(\mathbb{R}^{3})$ for each $0<\lambda <\lambda
^{\ast },$ which changes sign exactly once in $\mathbb{R}^{3}.$

\section{Proof of Theorem \protect\ref{t3}}

As in Section 5, we also give a precise description of the Palais--Smale
sequence for $I_{\lambda }$ at the beginning of this section.

\begin{proposition}
\label{d2}Suppose that $2<p<4,$ and conditions ${(F1)}-{(F2)},(K1)$ and $%
(K3) $ hold. Let $\left\{ \overline{u}_{n}\right\} \subset H^{1}(\mathbb{R}%
^{3})$ be a sequence satisfying\newline
$\left( i\right) \ dist\left( \overline{u}_{n},\mathbf{N}_{\lambda
}^{(1)}\right) \rightarrow 0;$\newline
$\left( ii\right) \ I_{\lambda }(\overline{u}_{n})\rightarrow \theta
_{\lambda }^{-};$\newline
$\left( iii\right) \ I_{\lambda }^{\prime }(\overline{u}_{n})=o(1)$ strongly
in $H^{-1}(\mathbb{R}^{3});$\newline
$\left( iv\right) $ $\left\vert \Phi _{\lambda }^{+}\left( \overline{u}%
_{n}\right) -1\right\vert +\left\vert \Phi _{\lambda }^{-}\left( \overline{u}%
_{n}\right) -1\right\vert \rightarrow 0.$\newline
Then there exist a subsequence $\{\overline{u}_{n}\}$ and $u_{\lambda }\in
\mathbf{N}_{\lambda }^{(1)}$ such that $\overline{u}_{n}\rightarrow
u_{\lambda }$ strongly in $H^{1}(\mathbb{R}^{3})$ for each $0<\lambda <%
\overline{\lambda }^{\ast }.$
\end{proposition}

\begin{proof}
The proof is analogous to that of Proposition \ref{d1}, and we omit it here.
\end{proof}

\textbf{We now begin to prove Theorem \ref{t3}:} By Corollary \ref{f5-2} and
Proposition \ref{d2}, for each $0<\lambda <\overline{\lambda }^{\ast },$ Eq.
$\left( E_{\lambda }\right) $ admits a nodal solution $u_{\lambda }$ such
that $I_{\lambda }\left( u_{\lambda }\right) =\theta _{\lambda }^{-}.$
Moreover, similar to the argument in \cite[Theorem 1.3]{AS}, $u_{\lambda }$
changes sign exactly once in $\mathbb{R}^{3}.$ Consequently, system $%
(SP_{\lambda })$ admits a nodal solution $(u_{\lambda },\phi _{K,u_{\lambda
}})\in H^{1}(\mathbb{R}^{3})\times D^{1,2}(\mathbb{R}^{3})$ for each $%
0<\lambda <\overline{\lambda }^{\ast },$ which changes sign exactly once in $%
\mathbb{R}^{3}.$

\section{Acknowledgments}

J. Sun is supported by the National Natural Science Foundation of China
(Grant No. 11671236). T.F. Wu is supported in part by the Ministry of
Science and Technology, Taiwan (Grant 106-2115-M-390-002-MY2), the
Mathematics Research Promotion Center, Taiwan and the National Center for
Theoretical Sciences, Taiwan.

\end{document}